\documentclass[11pt]{article}

\usepackage{xcolor}
\usepackage[numbers,sort&compress]{natbib}
\usepackage{setspace}
\PassOptionsToPackage{colorlinks=true,allcolors=blue}{hyperref}
\usepackage{definitions}
\usepackage{todonotes}
\newtheorem{theorem}[prop]{Theorem}
\pgfplotsset{compat=1.17}
\DeclareMathOperator{\proj}{proj}
\DeclareMathOperator{\prox}{prox}
\DeclareMathOperator*{\argmin}{arg\,min}
\newcommand{\R}{\RR}
\newcommand{\ip}[2]{\langle #1, #2\rangle}
\newcommand{\eps}{\varepsilon}

\title{cuPDQPx: A GPU-Accelerated Restarted Halpern PDHG for Large-Scale Convex Quadratic Programming}

\author{Haihao Lu\thanks{MIT, Sloan School of Management (haihao@mit.edu).} \and Zedong Peng\thanks{MIT, Sloan School of Management (zdpeng@mit.edu).} \and Jinwen Yang\thanks{University of Chicago, Department of Statistics (jinweny@uchicago.edu).}}
\date{}

\begin{document}
\maketitle

\begin{abstract}
We present cuPDQPx, a GPU-accelerated, factorization-free first-order solver for large-scale convex quadratic programming (QP).
Building on the restarted reflected Halpern primal--dual framework of cuPDLPx, we introduce an auxiliary-variable splitting that yields explicit updates without linearizing the quadratic objective or solving inner quadratic subproblems.
A preconditioned proximal point interpretation yields both sublinear and linear fixed-point residual convergence guarantees. 
For practical efficiency, cuPDQPx integrates diagonal preconditioning, step-size selection, adaptive restart, and adaptive primal and quadratic weight updates.
cuPDQPx is implemented in C and CUDA with optimized sparse matrix--vector products, kernel fusion, and CUDA Graphs.
Numerical experiments on the Maros--M\'esz\'aros collection, a convex subset of QPLIB, and synthetic large-scale QPs demonstrate that cuPDQPx has strong numerical performance at both $10^{-4}$ and $10^{-8}$ tolerances.
\end{abstract}

\section{Introduction}
\label{sec:introduction}

Convex quadratic programming (QP) is a fundamental class of
optimization problems with wide applications in machine learning,
control, finance, and operations research. Examples include support
vector machines, constrained least squares, model predictive control,
and portfolio optimization \citep{stellato2020osqp,lu2026pdqp}.
Classical approaches, including active-set and interior-point methods,
provide reliable solutions with high accuracy.
However, their computational bottleneck is solving linear systems,
which can become expensive on large instances.
Moreover, sparse matrix factorizations may introduce substantial
fill-in and require considerably more memory than the original
problem data. These limitations have motivated the development
of first-order methods for large-scale QP.

Recently, factorization-free first-order methods have emerged
as a promising approach to large-scale optimization.
The main computational cost of these methods lies in
matrix--vector multiplication rather than matrix factorization,
making them well suited to modern parallel architectures
such as graphics processing units (GPUs).
For linear programming (LP), PDLP and its GPU implementation \citep{applegate2021pdlp,lu2025cupdlpjl, lu2023cupdlpc}
demonstrates the effectiveness of combining restarted
primal--dual hybrid gradient (PDHG) with practical enhancements.
Building on restarted Halpern PDHG
\citep{lu2024halpern}, cuPDLPx \citep{lu2025cupdlpx}
combines algorithmic enhancements with an efficient GPU
implementation. For convex QP, solvers such as
PDQP \citep{lu2026pdqp}, PDHCG \citep{huang2025pdhcg}, and
HPR-QP \citep{chen2025hprqp} combine factorization-free primal--dual methods with
GPU acceleration and have demonstrated promising
performance on standard benchmarks and large-scale instances.

The main computational obstacle in extending LP-oriented PDHG solvers to QP is the
quadratic objective: a direct application can replace the inexpensive box projections
of LP with coupled quadratic subproblems (see Section \ref{sec:vanilla-pdhg} for a more detailed discussion).
A unifying perspective is that splitting methods such as PDHG and ADMM are
instances of the preconditioned proximal point method \citep{lu2023unified}.
The choice of splitting and proximal metric determines the subproblems and computational
cost of each iteration. This perspective motivates a splitting that retains the quadratic
objective while yielding inexpensive explicit updates.

In this paper, we propose a new PDHG-based method
for convex QP based on an auxiliary-variable
saddle-point formulation.
This equivalent reformulation yields explicit updates without linearizing the quadratic
objective or requiring an inner subproblem solve. We combine this splitting with the
restarted reflected Halpern framework underlying cuPDLPx \citep{lu2025cupdlpx}.
We implement the proposed method in cuPDQPx,
a GPU-accelerated QP solver.
The contributions of this paper are summarized as follows.

\begin{enumerate}
    \item \textbf{A restarted Halpern primal--dual method for convex QP.}
    We develop an auxiliary-variable splitting with explicit
    updates for general convex QP.
    We establish a preconditioned proximal point interpretation
    and a sublinear fixed-point residual bound for the reflected
    Halpern iteration under fixed parameters and a fixed anchor. We further establish linear convergence of the restart anchors under a residual-based adaptive restart rule with fixed parameters.

    \item \textbf{Practical algorithmic enhancements.}
    We incorporate preconditioning, adaptive restart, and
    adaptive primal and quadratic weights to balance the
    primal, auxiliary, and constraint-dual blocks.
    We derive movement-based target weights from the
    proximal metric, accounting for the dependence of
    the step size on the quadratic weight.

    \item \textbf{Explicit treatment of diagonal QP.}
    We describe the diagonal-$Q$ specialization, where the primal
    proximal update has a coordinatewise closed form and PDHG
    applies directly without an auxiliary variable.
    We detail the simplified algorithmic enhancements and provide
    a dedicated implementation that is selected automatically
    when $Q$ is diagonal.

    \item \textbf{Efficient GPU implementation.}
    We implement the proposed method in CUDA C using
    optimized sparse matrix--vector products, kernel fusion,
    and CUDA Graphs.
    Experiments on the Maros--M\'esz\'aros collection,
    a convex subset of QPLIB, and synthetic QPs demonstrate
    competitive performance relative to existing CPU and GPU
    first-order solvers.
    On the synthetic set, cuPDQPx achieves the lowest
    shifted geometric mean of running time among the
    tested solvers at tolerances $10^{-4}$ and $10^{-8}$.
\end{enumerate}

To motivate the proposed splitting, we first introduce the QP formulation and notation,
then examine the computational challenges of applying PDHG to QP and existing approaches
to these challenges. We also review related approaches.

\subsection{Problem formulation and notation}
\label{sec:qp-formulation}
We consider the convex quadratic program
\begin{equation}
\label{eq:qp}
\begin{aligned}
\min_{x\in\R^n}\quad & \tfrac12 x^\top Q x + c^\top x\\
\text{s.t.}\quad & \ell_c \le Ax \le u_c,\\
& \ell_v \le x \le u_v,
\end{aligned}
\end{equation}
where $x\in\R^n$ is the decision variable. The objective is defined by a symmetric positive
semidefinite matrix $Q\in\R^{n\times n}$ and a vector $c\in\R^n$. The constraints are specified
by the matrix $A\in\R^{m\times n}$, the constraint bounds $\ell_c,u_c$, and the variable bounds
$\ell_v,u_v$.
For notational convenience, we denote the variable box by
$X:=\{x: \ell_v\le x\le u_v\}$, with
$\ell_v\in(\R\cup\{-\infty\})^n$ and $u_v\in(\R\cup\{+\infty\})^n$,
and the constraint box by $S:=\{s:\ell_c\le s\le u_c\}$, with
$\ell_c\in(\R\cup\{-\infty\})^m$ and $u_c\in(\R\cup\{+\infty\})^m$.

A Wolfe-type dual formulation of~\eqref{eq:qp}
\citep{wolfe1961duality} is given by
\begin{equation}
\label{eq:dual}
\begin{aligned}
\max_{x\in\R^n,\,y\in Y,\,r\in R}\quad
& -\tfrac12 x^\top Qx - p(-y;\ell_c,u_c) - p(-r;\ell_v,u_v)\\
\text{s.t.}\quad & Qx+c-A^\top y-r=0,
\end{aligned}
\end{equation}
where $p$ is the support function of $[\ell,u]$, defined by
\begin{equation}
\label{eq:pfun}
p(r;\ell,u):= u^\top r^+ - \ell^\top r^-,\qquad r^+=\max\{r,0\},\ r^-=\max\{-r,0\},
\end{equation}
with the convention $0\cdot(\pm\infty)=0$. This dual is used in the termination check.

The variables $y$ and $r$ are the multipliers associated with the constraints
$\ell_c\le Ax\le u_c$ and the variable bounds $\ell_v\le x\le u_v$, respectively.
The sets $Y:=\prod_{i=1}^m Y_i\subseteq\R^m$ and
$R:=\prod_{i=1}^n R_i\subseteq\R^n$ encode the multiplier sign restrictions, with
\begin{equation}
\label{eq:dual-domains}
Y_i:=\begin{cases}
\{0\}, & \ell_{c,i}=-\infty,\ u_{c,i}=+\infty,\\
\R^-, & \ell_{c,i}=-\infty,\ u_{c,i}\in\R,\\
\R^+, & \ell_{c,i}\in\R,\ u_{c,i}=+\infty,\\
\R, & \text{otherwise},
\end{cases}
\qquad
R_i:=\begin{cases}
\{0\}, & \ell_{v,i}=-\infty,\ u_{v,i}=+\infty,\\
\R^-, & \ell_{v,i}=-\infty,\ u_{v,i}\in\R,\\
\R^+, & \ell_{v,i}\in\R,\ u_{v,i}=+\infty,\\
\R, & \text{otherwise}.
\end{cases}
\end{equation}
Here, $\R^+$ and $\R^-$ denote the nonnegative and nonpositive real numbers, respectively.

Equivalently, the primal problem~\eqref{eq:qp} can be expressed in saddle-point form as
\begin{equation}
\label{eq:saddle}
\min_{x\in X}\ \max_{y\in Y}\ \mathcal{L}(x,y):=\tfrac12 x^\top Qx + c^\top x - y^\top Ax - p(-y;\ell_c,u_c),
\end{equation}
whose optimality conditions are
\begin{equation}
\label{eq:kkt}
Qx + c - A^\top y - r = 0,\qquad -r\in N_X(x),\qquad -y\in N_S(Ax),
\end{equation}
where $N_X$ and $N_S$ denote the outward normal cones of $X$ and $S$, respectively.

\subsection{PDHG for QP: Challenges and Existing Approaches}
\label{sec:vanilla-pdhg}

To examine the cost of applying PDHG directly to~\eqref{eq:saddle}, we first recall
the method \citep{chambolle2011first} for convex--concave saddle-point problems of the form
\[
\min_x\ \max_y\quad f(x)+\ip{Kx}{y}-g(y),
\]
where $f,g$ are proper closed convex functions and $K$ is linear.
For step sizes $\tau,\sigma>0$ and symmetric positive semidefinite proximal matrices
$M_x,M_y$, write $\norm{u}_M^2=u^\top Mu$. The PDHG updates are
\begin{equation}
\label{eq:general-pdhg}
\begin{aligned}
x^{k+1}&\in\argmin_x\left\{f(x)+\ip{Kx}{y^k}
+\tfrac{1}{2\tau}\norm{x-x^k}_{M_x}^2\right\},\\
y^{k+1}&\in\argmin_y\left\{g(y)-\ip{K(2x^{k+1}-x^k)}{y}
+\tfrac{1}{2\sigma}\norm{y-y^k}_{M_y}^2\right\}.
\end{aligned}
\end{equation}

Applying these updates to~\eqref{eq:saddle} directly with $M_x=M_y=I$ gives
\begin{subequations}
\label{eq:pdhg-prox}
\begin{align}
x^{k+1} &= \argmin_{x\in X}\left\{\tfrac12 x^\top Qx+c^\top x-\ip{y^k}{Ax}
+\tfrac{1}{2\tau}\norm{x-x^k}_2^2\right\}, \label{eq:proxsub}\\
y^{k+1} &= y^k-\sigma A(2x^{k+1}-x^k)
+\sigma\proj_S\!\big(A(2x^{k+1}-x^k)-\sigma^{-1}y^k\big). \label{eq:pdhg-dual}
\end{align}
\end{subequations}
If a saddle point exists, fixed step sizes satisfying
$\tau\sigma\norm{A}_2^2<1$ ensure convergence to a saddle point of~\eqref{eq:saddle}.
For LP ($Q=0$), the primal update~\eqref{eq:proxsub} reduces to a box projection.
For nondiagonal $Q$, the quadratic term couples the primal coordinates, and the same
update generally requires solving a quadratic subproblem. When $X=\R^n$, its optimality
condition is the linear system
\[
(I+\tau Q)x^{k+1}=x^k+\tau(A^\top y^k-c).
\]
With variable bounds, it is a box-constrained QP. Solving these subproblems can require
matrix factorizations or inner iterations, increasing the per-iteration cost relative to LP.
The dual update~\eqref{eq:pdhg-dual} still uses a projection onto the constraint box $S$.
When $Q$ is diagonal, the primal update also separates across coordinates and admits a
closed form; we develop this specialization in Section~\ref{app:diagonal}.

Existing factorization-free GPU solvers handle the quadratic objective in different ways.
PDQP \citep{lu2026pdqp} linearizes the quadratic term in~\eqref{eq:proxsub} at a weighted
combination of the current primal iterate and its running average, and combines the resulting
explicit gradient update with acceleration and restarting.
PDHCG \citep{huang2025pdhcg} solves~\eqref{eq:proxsub} inexactly, using conjugate gradients
when no variable bound is finite and projected gradient with Barzilai--Borwein steps otherwise.
HPR-QP \citep{chen2025hprqp} uses a different splitting framework based on a certain dual
formulation and Halpern Peaceman--Rachford iteration.
In Section~\ref{sec:framework}, we instead introduce an auxiliary-variable splitting that
yields explicit updates without linearizing the quadratic objective or solving~\eqref{eq:proxsub}.

\subsection{Related literature}
\label{sec:related}

The preceding discussion focuses on the treatment of the quadratic objective.
We now review the broader operator-splitting frameworks and convergence results
relevant to cuPDQPx.

\paragraph{ADMM-based solvers.}
OSQP \citep{stellato2020osqp} and SCS
\citep{odonoghue2016scs} are widely used first-order
solvers based on operator splitting.
OSQP applies ADMM directly to convex QPs, while SCS
applies ADMM to a homogeneous self-dual embedding of
conic programs.
Both methods require solving a structured linear system at each iteration.
There are two standard approaches to these solves:
direct methods compute and reuse a matrix factorization
when the system matrix remains unchanged, while
indirect methods use iterative linear solvers such
as conjugate gradients.
The direct approach benefits from factorization reuse
but may incur substantial memory costs.
The indirect approach avoids matrix factorization
but requires additional matrix--vector products
within each outer iteration.
In contrast, cuPDQPx does not require a linear-system solve and 
each iteration consists only of sparse matrix--vector products
and vector operations.

\paragraph{PDHG-based solvers.}
Recent progress on PDHG-based solvers has been driven 
in part by developments in large-scale LP.
PDLP \citep{applegate2021pdlp} combines restarted averaged
PDHG with practical enhancements including preconditioning,
adaptive step sizes, and primal-weight adjustment.
Subsequent work developed restarted Halpern PDHG and its
reflected variant \citep{lu2024halpern}, which provide
accelerated linear convergence guarantees for LP.
Building on this development, cuPDLPx
\citep{lu2025cupdlpx} employs restarted reflected Halpern
PDHG together with adaptive restart and primal-weight
updates in an efficient GPU implementation.
For convex QP, the restarted accelerated PDHG method underlying
PDQP \citep{lu2026pdqp} achieves a linear convergence
rate matching an information-theoretic lower bound for
a broad class of primal--dual first-order methods.
PDHCG-CQP \citep{li2026pdhcgcqp} extends PDHCG \citep{huang2025pdhcg} to conic QP,
including second-order, exponential, and power cone constraints.

\paragraph{Halpern Peaceman--Rachford-based solvers.}
HPR-QP \citep{chen2025hprqp} develops a dual
Halpern Peaceman--Rachford method for large-scale
convex composite QP.
It builds on a restricted Wolfe dual and uses a
symmetric Gauss--Seidel technique to simplify the
subproblem updates.
HPR-QP is closely related to the present work through
its use of Halpern-type fixed-point iterations for
large-scale convex QP.
cuPDQPx applies restarted reflected Halpern iteration to the PDHG operator
derived from its auxiliary-variable saddle-point formulation.
Additionally, the connection between HPR-LP \citep{chen2026hprlp} and the restarted Halpern
PDHG underlying cuPDLPx is discussed in
\citep{lu2025cupdlpx}.

\subsection{Paper organization}

The remainder of this paper is organized as follows.
Section~\ref{sec:framework} presents the proposed
splitting, its proximal point interpretation, and
the reflected Halpern residual bound.
Section~\ref{sec:enhancements} describes the practical
algorithmic enhancements.
Section~\ref{app:diagonal} specializes the method to diagonal quadratic objectives.
Section~\ref{sec:experiments} presents the GPU
implementation and computational experiments.
Section~\ref{sec:conclusion} concludes the paper.

\section{Auxiliary-Variable Splitting and Restarted Halpern PDHG}
\label{sec:framework}

We now develop an auxiliary-variable splitting of~\eqref{eq:saddle} that avoids the
quadratic subproblem~\eqref{eq:proxsub}. We first derive explicit primal--dual updates
and their preconditioned proximal point interpretation, then apply restarted reflected
Halpern iteration and establish a fixed-point residual bound.

\subsection{Primal--dual splitting scheme}
\label{sec:scheme}

Since $Q\succeq0$, introducing an auxiliary variable $w\in\R^n$ gives the representation
\[
\tfrac12 x^\top Qx
=\max_{w\in\R^n}\left\{x^\top Qw-\tfrac12 w^\top Qw\right\}.
\]
The maximum is attained whenever $Qw=Qx$, and in particular at $w=x$.
Thus~\eqref{eq:saddle} is equivalent to
\begin{equation}
\label{eq:augmented-saddle}
\min_{x\in X}\ \max_{w\in\R^n,\,y\in Y}\quad
\widetilde{\mathcal{L}}(x,w,y):=c^\top x+x^\top Qw
-\tfrac12 w^\top Qw-y^\top Ax-p(-y;\ell_c,u_c).
\end{equation}
Let $\tau,\sigma,\beta>0$ be the step sizes for $x$, $y$, and $w$, respectively.
We apply the PDHG updates~\eqref{eq:general-pdhg} to~\eqref{eq:augmented-saddle}, treating $x$ as the primal
variable and $(w,y)$ as the dual block, with proximal matrices $I$,
$Q$, and $I$ for $x$, $w$, and $y$, respectively.

In particular, writing $\norm{v}_Q^2=v^\top Qv$,
the dual update over $w$ solves
\[
w^{k+1}\in\argmin_{w\in\R^n}
\left\{\tfrac12 w^\top Qw-(2x^{k+1}-x^k)^\top Qw
+\tfrac{1}{2\beta}\norm{w-w^k}_Q^2\right\}.
\]
Its optimality condition is $Q((1+\beta)w^{k+1}-w^k-\beta (2x^{k+1}-x^k))=0$,
which is satisfied by the explicit averaging update below, including when $Q$ is singular.
The primal proximal step gives a projection onto $X$. For the $y$-update, let
$h(y):=p(-y;\ell_c,u_c)$ and $v^k:=y^k-\sigma A(2x^{k+1}-x^k)$.
Completing the square in the dual proximal subproblem gives
\[
y^{k+1}=\argmin_{y\in\R^m}
\left\{h(y)+\tfrac{1}{2\sigma}\norm{y-v^k}_2^2\right\}
=\prox_{\sigma h}(v^k).
\]
Since $h$ is the support function of $-S$, Moreau's identity yields
\[
\prox_{\sigma h}(v)=v-\sigma\proj_{-S}(v/\sigma)
=v+\sigma\proj_S(-v/\sigma).
\]
Substituting $v=v^k$ gives the $y$-update below; the domain of $h$ already encodes $Y$.
Thus, the PDHG updates on \eqref{eq:augmented-saddle} can be written as
\begin{subequations}
\label{eq:iter}
\begin{align}
x^{k+1} &= \proj_X\!\big(x^k-\tau\,(c + Qw^k - A^\top y^k)\big), \label{eq:iter-x}\\
w^{k+1} &= \frac{1}{1+\beta}\,w^k+\frac{\beta}{1+\beta}\,(2x^{k+1}-x^k). \label{eq:iter-w}\\
y^{k+1} &= y^k-\sigma\Big(A(2x^{k+1}-x^k)
-\proj_{S}\big(A(2x^{k+1}-x^k)-\sigma^{-1}y^k\big)\Big), \label{eq:iter-y}
\end{align}
\end{subequations}
The primal update~\eqref{eq:iter-x} is a box projection and requires no inner solver.
With $z^k=(x^k,w^k,y^k)$, we write $z^{k+1}=T(z^k)$ to denote one preconditioned PDHG step
defined by~\eqref{eq:iter}.
Here, $w$ is an auxiliary variable used internally by the algorithm, which returns only the
primal--dual pair $(x,y)$ for the original QP.

\paragraph{Proximal point interpretation.}
To derive a step-size condition and establish firm nonexpansiveness of $T$, we express
the iteration~\eqref{eq:iter} as a preconditioned proximal point step.
For $z=(x,w,y)$, the first-order saddle-point optimality conditions associated
with~\eqref{eq:augmented-saddle} are $0\in\mathcal{F}(z)$, where
\begin{equation}
\label{eq:augmented-operator}
\mathcal{F}(z)=\mathcal{F}(x,w,y)=
\begin{pmatrix}
c+Qw-A^\top y+\partial\iota_X(x)\\
Q(w-x)\\
Ax+\partial[p(-\cdot;\ell_c,u_c)](y)
\end{pmatrix},
\end{equation}
where $\iota_X$ is zero on $X$ and $+\infty$ elsewhere, and $\partial$ denotes the convex
subdifferential.

Following the proximal point interpretation of PDHG in \cite[Section~2]{lu2023unified},
the updates~\eqref{eq:iter} satisfy
\begin{equation}
\label{eq:ppm}
P(\tau,\sigma,\beta)\,(z^k-z^{k+1})\in\mathcal{F}(z^{k+1}),
\end{equation}
with the symmetric matrix
\begin{equation}
\label{eq:P}
P(\tau,\sigma,\beta)=
\begin{pmatrix}
\frac{1}{\tau}I & -Q & A^\top\\[2pt]
-Q & \frac{1}{\beta}Q & 0\\[2pt]
A & 0 & \frac{1}{\sigma}I
\end{pmatrix}.
\end{equation}

The convergence analysis in Section~\ref{sec:halpern} relies on the proximal point
relation~\eqref{eq:ppm} under the stated step-size conditions and allows $P$ to be positive semidefinite.
If $Q$ is singular, then $P$ is necessarily singular, and~\eqref{eq:ppm} leaves the component
of $w^{k+1}$ in $\ker Q$ undetermined. The explicit update~\eqref{eq:iter-w} selects this
component while ensuring that~\eqref{eq:ppm} holds; uniqueness of the $w$-subproblem solution
is not required for the subsequent analysis.
The fixed points of $T$ are exactly $(x,x,y)$ with $(x,y)$ a saddle point
of~\eqref{eq:saddle}.
Since $w$ is an auxiliary variable and the algorithm returns only $(x,y)$, singularity of $Q$
does not preclude convergence of the returned primal--dual variables under the stated step-size
conditions. We discuss this point in more detail in Section~\ref{sec:halpern}.

Proposition~\ref{prop:psd} characterizes when $P\succeq0$ and establishes firm
nonexpansiveness in the induced seminorm.

\begin{prop}
\label{prop:psd}
Let $Q\succeq0$ and $\tau,\sigma,\beta>0$. Then $P(\tau,\sigma,\beta)\succeq0$ if and only if
\begin{equation}
\label{eq:stepcond}
\big\|\tau\beta Q+\tau\sigma A^\top A\big\|_2\ \le\ 1 .
\end{equation}
Under this condition, the update map $T$ defined by~\eqref{eq:iter} is firmly nonexpansive
in the seminorm $\norm{v}_P:=\sqrt{\ip{v}{Pv}}$; that is, for all $z_1,z_2$,
\[
\norm{T(z_1)-T(z_2)}_P^2
\ \le\ \ip{T(z_1)-T(z_2)}{P(z_1-z_2)}.
\]
\end{prop}

\begin{proof}
The range condition for the generalized Schur complement of the $(w,w)$ block holds
because the coupling blocks $-Q$ and $0$ have ranges contained in $\mathrm{range}(Q)$.
Taking this complement and then the Schur complement of the $(y,y)$ block gives
\[
P\succeq0\quad\Longleftrightarrow\quad
\frac{1}{\tau}I-\beta Q-\sigma A^\top A\succeq0.
\]
Since $\tau\beta Q+\tau\sigma A^\top A\succeq0$, this is equivalent to~\eqref{eq:stepcond}.

The convex--concave structure of~\eqref{eq:augmented-saddle} implies that $\mathcal{F}$ is
monotone. Applying monotonicity to~\eqref{eq:ppm} at $z_1$ and $z_2$ gives the stated
firm nonexpansiveness inequality.
\end{proof}

For $Q=0$, condition~\eqref{eq:stepcond} reduces to the LP condition
$\tau\sigma\norm{A}_2^2\le1$.
The firm nonexpansiveness of $T$ is used in the analysis of the reflected Halpern scheme
in Section~\ref{sec:halpern}.

For diagonal $Q$, the primal subproblem~\eqref{eq:proxsub} admits a closed-form
solution, so no auxiliary variable is needed. The corresponding updates and
algorithmic enhancements are detailed in Section~\ref{app:diagonal}.

\subsection{Restarted Reflected Halpern scheme}
\label{sec:halpern}

Following the restarted reflected Halpern PDHG scheme in \citep{lu2024halpern}, also used in
cuPDLPx \citep{lu2025cupdlpx}, we apply Halpern iteration to the reflected operator $2T-I$,
with the initial point $z^{0}$ as the fixed anchor:
\begin{equation}
\label{eq:halpern}
z^{k+1}=\frac{k+1}{k+2}\Big(2T(z^k)-z^k\Big)+\frac{1}{k+2}\,z^{0},\qquad k=0,1,2,\dots,
\end{equation}
where $T$ denotes one step of~\eqref{eq:iter}.
Reflection is applied to all three blocks, including $w$ and $Qw$.

The following theorem applies the standard reflected Halpern analysis to this update map.

\begin{theorem}
\label{thm:rate}
Assume $Q\succeq0$ and that~\eqref{eq:qp} has a KKT point.
Let $T$ be defined by~\eqref{eq:iter} with fixed
$\tau,\sigma,\beta>0$ satisfying~\eqref{eq:stepcond}, and let $P$ be the matrix
in~\eqref{eq:P}. Then $T$ and $2T-I$ are nonexpansive in the seminorm $\norm{\cdot}_P$,
the fixed points of $T$ coincide with the KKT points of~\eqref{eq:qp} (in the $(x,y)$
components), and the iterates generated by the reflected Halpern update~\eqref{eq:halpern} satisfy
\begin{equation}
\label{eq:rate}
\norm{z^{k}-T(z^{k})}_P\ \le\ \frac{\norm{z^{0}-z^\star}_P}{k+1}
\end{equation}
for every fixed point $z^\star$ of $T$.
\end{theorem}

\begin{proof}
Proposition~\ref{prop:psd} gives $P\succeq0$ and firm nonexpansiveness
of $T$ in $\norm{\cdot}_P$.
Hence $T$ and $2T-I$ are nonexpansive in $\norm{\cdot}_P$.
If $u-v\in\ker P$, nonexpansiveness implies $T(u)-T(v)\in\ker P$.
Thus both maps induce well-defined nonexpansive maps on the quotient of the iterate space
by $\ker P$. Equipped with the inner product induced by $P$, this quotient is a
finite-dimensional Hilbert space.
The standard Halpern bound for nonexpansive maps \citep{lieder2021halpern} applies on this
space to the induced reflected map, with anchor $z^0+\ker P$ and fixed point
$z^\star+\ker P$. Since the quotient norm equals the $P$-seminorm of any representative and
$\norm{(2T-I)(z^k)-z^k}_P=2\norm{T(z^k)-z^k}_P$, the same argument as in
\citep[Section~6, Corollary~1]{lu2024halpern} gives~\eqref{eq:rate}.
The fixed-point characterization established above for~\eqref{eq:iter}
gives the claimed correspondence with the KKT points.
\end{proof}

\paragraph{Restarting.}
We restart the reflected Halpern iteration with a new anchor whenever a restart condition is
satisfied \citep{lu2024halpern,lu2025cupdlpx}. Let $n$ index restart windows and $k$ index iterations
within a window. With anchor $z^{n,0}$, the update~\eqref{eq:halpern} becomes
\begin{equation}
\label{eq:restarted-halpern}
z^{n,k+1}=\frac{k+1}{k+2}\Big(2T(z^{n,k})-z^{n,k}\Big)
+\frac{1}{k+2}\,z^{n,0}.
\end{equation}
If a restart is triggered at inner iterate $z^{n,k_n}$, the next window starts from the
corresponding non-reflected PDHG point:
\begin{equation}
\label{eq:restart-anchor}
z^{n+1,0}:=T(z^{n,k_n}).
\end{equation}
This point serves as both the new anchor and the initial iterate, and the inner counter is
reset to zero.
Theorem~\ref{thm:rate} applies within each restart window, where the parameters and anchor remain fixed and $k$ counts iterations from the anchor.
The adaptive restart conditions and parameter updates are described in
Sections~\ref{sec:restart} and~\ref{sec:weights}, respectively.

Following \citep{lu2024halpern}, we consider the adaptive
restart rule
\begin{equation}
\label{eq:theoretical-restart}
\norm{z^{n,k}-T(z^{n,k})}_P
\le e^{-1}\norm{z^{n,0}-T(z^{n,0})}_P,
\end{equation}
where the first restart window has an arbitrary finite
length $k_0$, and each subsequent window restarts at the
first $k=k_n\ge1$ satisfying~\eqref{eq:theoretical-restart}.
A window terminates if its anchor has zero fixed-point
residual.

\begin{theorem}
\label{thm:linear-restart}
Assume $Q\succeq0$ and that~\eqref{eq:qp}
has a KKT point. Let $T$ be defined by~\eqref{eq:iter}
with fixed $\tau,\sigma,\beta>0$ satisfying
\(\norm{\tau\beta Q+\tau\sigma A^\top A}_2<1\).
Set $\mathcal Z^*:=\operatorname{Fix}(T)$ and
$\operatorname{dist}_P(z,\mathcal Z^*)
:=\inf_{z^\star\in\mathcal Z^*}\norm{z-z^\star}_P$.
Then there exists $\alpha>0$, depending on the problem
data, parameters, and initial point, such that the
restarted iteration~\eqref{eq:restarted-halpern}--
\eqref{eq:restart-anchor} with
\eqref{eq:theoretical-restart} satisfies
\begin{equation}
\label{eq:linear-restart}
\begin{aligned}
k_n\le\left\lceil\frac{e}{\alpha}\right\rceil, \ 
\operatorname{dist}_P(z^{n,0},\mathcal Z^*)\le
\frac{e^{-(n-1)}}{\alpha(k_0+1)}
\operatorname{dist}_P(z^{0,0},\mathcal Z^*),
\qquad n\ge1,
\end{aligned}
\end{equation}
for every nonterminal restart window. Thus the restart anchors converge linearly, with
$O(\alpha^{-1}\log(1/\eps))$ iteration complexity.
\end{theorem}

\begin{rem}
{We would like to comment that convergence in the $P$-seminorm guarantees convergence of the primal-dual variables to the
optimal solution set, even when $Q$ is singular. To see this, under the strict step-size condition in
Theorem~\ref{thm:linear-restart}, we have
\(
\ker P=\{(0,v,0):Qv=0\}.
\)
Let $\mathcal Z^*_{xy}$ denote the set of primal--dual
optimal pairs of~\eqref{eq:qp}, and let $\lambda_+>0$
be the smallest positive eigenvalue of $P$.
For any $z=(x,w,y)$,
\(
\operatorname{dist}_2((x,y),\mathcal Z^*_{xy})
\le \frac{1}{\sqrt{\lambda_+}}
\operatorname{dist}_P(z,\mathcal Z^*),
\)
and consequently,
$\operatorname{dist}_P(z,\mathcal Z^*)=0$
implies $(x,y)\in\mathcal Z^*_{xy}$.}
\end{rem}

\begin{proof}
Fix $z^\star\in\mathcal Z^*$ and let
$R:=\norm{z^{0,0}-z^\star}_P$.
By Proposition~\ref{prop:psd}, $T$ and $2T-I$ are
nonexpansive in $\norm{\cdot}_P$. Hence, for all $n,k$,
\[
\norm{z^{n,k}-z^\star}_P\le R,\qquad
\norm{T(z^{n,k})-z^\star}_P\le R \ .
\]

Let $\mathcal N:=\ker P$ and
$\Pi:=\proj_{\mathcal N^\perp}$.
Completing the square in~\eqref{eq:P} gives
\[
\norm{z}_P^2
=\norm{x}_{M}^2
+\beta^{-1}\norm{w-\beta x}_Q^2
+\sigma^{-1}\norm{y+\sigma Ax}_2^2,
\]
where
$M:=\tau^{-1}I-\beta Q-\sigma A^\top A\succ0$.
Consequently,
\(
\mathcal N=\{(0,v,0):Qv=0\},
\ P|_{\mathcal N^\perp}\succ0.
\)
Let $\lambda_+>0$ be the smallest positive eigenvalue
of $P$. Then
\[
\norm{\Pi T(z^{n,k})-\Pi z^\star}_2
\le\frac{R}{\sqrt{\lambda_+}},
\]
so all projected PDHG points belong to the compact set
\[
\mathcal B:=
\left\{u\in\mathcal N^\perp:
\norm{u-\Pi z^\star}_2
\le\frac{R}{\sqrt{\lambda_+}}\right\}.
\]

Let $\mathcal Z_{\mathcal F}:=\mathcal F^{-1}(0)$.
Since $\mathcal F$ is a polyhedral multifunction,
Robinson's error bound \citep{robinson1981polyhedral}
gives $\kappa_0,\rho>0$ such that
\[
\operatorname{dist}_2(u,\mathcal Z_{\mathcal F})
\le\kappa_0\operatorname{dist}_2(0,\mathcal F(u))
\]
whenever $\operatorname{dist}_2(0,\mathcal F(u))<\rho$.
Moreover, $\Pi z^\star\in\mathcal Z_{\mathcal F}$, so
\(
\operatorname{dist}_2(u,\mathcal Z_{\mathcal F})
\le\frac{R}{\sqrt{\lambda_+}},
\ u\in\mathcal B.
\)
Thus, setting
\(
\kappa:=
\max\left\{\kappa_0,
\frac{D}{\rho\sqrt{\lambda_+}}\right\},
\)
we obtain the uniform error bound
\begin{equation}
\label{eq:qp-error-bound}
\operatorname{dist}_2(u,\mathcal Z_{\mathcal F})
\le\kappa\operatorname{dist}_2(0,\mathcal F(u)),
\quad
u\in\mathcal B\cap\operatorname{dom}\mathcal F.
\end{equation}

Notice that
$\mathcal F(z)=\mathcal F(\Pi z)$ and
$\Pi\mathcal Z_{\mathcal F}=\Pi\mathcal Z^*$.
Using~\eqref{eq:ppm} and~\eqref{eq:qp-error-bound},
for every generated iterate $z$ we have
\[
\begin{aligned}
\operatorname{dist}_P(z,\mathcal Z^*)
&\le \norm{z-T(z)}_P+\operatorname{dist}_P(T(z),\mathcal Z^*)\\
&\le \norm{z-T(z)}_P+\sqrt{L}\,
  \operatorname{dist}_2(\Pi T(z),\mathcal Z_{\mathcal F})\\
&\le \norm{z-T(z)}_P+\kappa\sqrt{L}\,
  \norm{P(z-T(z))}_2\\
&\le (1+\kappa L)\norm{z-T(z)}_P,
\end{aligned}
\]
where $L:=\lambda_{\max}(P)$ and the last inequality
follows from $P^2\preceq LP$.
Set $\alpha:=(1+\kappa L)^{-1}>0$. Then
\begin{equation}
\label{eq:qp-sharpness}
\alpha \operatorname{dist}_P(z,\mathcal Z^*)\le \norm{z-T(z)}_P.
\end{equation}

By Theorem~\ref{thm:rate} and
\eqref{eq:qp-sharpness},
\[
\norm{z^{n,k}-T(z^{n,k})}_P
\le\frac{\operatorname{dist}_P(z^{n,0},\mathcal Z^*)}{k+1}
\le\frac{\norm{z^{n,0}-T(z^{n,0})}_P}{\alpha(k+1)}.
\]
Hence the restart condition is satisfied within
$k_n\le\lceil e/\alpha\rceil$ iterations.

Finally, nonexpansiveness of $T$ yields
$\norm{T(z)-T(T(z))}_P\le \norm{z-T(z)}_P$. Thus
\[
\begin{aligned}
&\norm{z^{n+1,0}-T(z^{n+1,0})}_P=\norm{T(z^{n,k_n})-T(T(z^{n,k_n}))}_P \le \norm{z^{n,k_n}-T(z^{n,k_n})}_P \le e^{-1}\norm{z^{n,0}-T(z^{n,0})}_P,
\end{aligned}
\]
and $\norm{z^{1,0}-T(z^{1,0})}_P
\le \norm{z^{0,k_0}-T(z^{0,k_0})}_P
\le \frac{\operatorname{dist}_P(z^{0,0},\mathcal Z^*)}{k_0+1}$.
Applying~\eqref{eq:qp-sharpness} recursively gives
\[
\begin{aligned}
\operatorname{dist}_P(z^{n,0},\mathcal Z^*)\le\alpha^{-1}\norm{z^{n,0}-T(z^{n,0})}_P\le\frac{e^{-(n-1)}}{\alpha(k_0+1)}
\operatorname{dist}_P(z^{0,0},\mathcal Z^*),
\end{aligned}
\]
which proves~\eqref{eq:linear-restart}.
\end{proof}

\section{Algorithmic Enhancements}
\label{sec:enhancements}

cuPDQPx combines the method introduced in Section~\ref{sec:framework} with preconditioning,
step-size selection, adaptive restart, and weight updates. We focus on the auxiliary-variable
scheme in this section; the corresponding enhancements for diagonal $Q$ are presented in
Section~\ref{app:diagonal}. We describe these enhancements, then present the termination
criteria and summarize the complete algorithm.
Default parameter values are collected in Section~\ref{sec:impl}.

\subsection{Preconditioning}
\label{sec:precond}

Following PDLP \citep{applegate2021pdlp,lu2025cupdlpjl} and cuPDLPx \citep{lu2025cupdlpx}, cuPDQPx applies diagonal preconditioning before the PDHG iterations to improve numerical conditioning. 
With positive diagonal row and column scaling matrices $D_1\in\R^{m\times m}$ and $D_2\in\R^{n\times n}$, the scaled problem data are given by
\[
\begin{aligned}
\tilde A &= D_1AD_2, & \tilde Q &= D_2QD_2, & \tilde c &= D_2c,\\
\tilde\ell_c &= D_1\ell_c, & \tilde u_c &= D_1u_c,\\
\tilde\ell_v &= D_2^{-1}\ell_v, & \tilde u_v &= D_2^{-1}u_v.
\end{aligned}
\]

Each scaling pass computes diagonal factors $\widehat D_1,\widehat D_2$ from the current
matrices $A,Q$. We initialize $D_1,D_2$ to the identity and update them as
\[
D_1\leftarrow\widehat D_1D_1,\qquad D_2\leftarrow D_2\widehat D_2.
\]

\paragraph[L-infinity Ruiz equilibration.]{$L_\infty$ Ruiz equilibration.}
Each iteration of Ruiz equilibration uses
\[
(\widehat D_1)_{ii}=\norm{A_{i,:}}_\infty^{-1/2},\qquad
(\widehat D_2)_{jj}=\max\!\left\{\norm{A_{:,j}}_\infty,
\norm{Q_{:,j}}_\infty\right\}^{-1/2}.
\]

\paragraph[L2 rescaling.]{$L_2$ rescaling.}
cuPDQPx next applies $L_2$ scaling with
\[
(\widehat D_1)_{ii}=\left(\sum_j|A_{ij}|^2\right)^{-1/4},\qquad
(\widehat D_2)_{jj}=\left(\sum_i|A_{ij}|^2+\sum_i|Q_{ij}|^2\right)^{-1/4}.
\]

\paragraph{Pock--Chambolle scaling.}
cuPDQPx finally applies Pock--Chambolle scaling \citep{pock2011diagonal} with exponent $\alpha$:
\[
(\widehat D_1)_{ii}=\left(\sum_j|A_{ij}|^\alpha\right)^{-1/2},\qquad
(\widehat D_2)_{jj}=\left(\sum_i|A_{ij}|^{2-\alpha}\right)^{-1/2}.
\]
The column scaling uses $A$ only.

\subsection{Step size}
\label{sec:stepselection}

Following PDLP \citep{applegate2021pdlp} and cuPDLPx \citep{lu2025cupdlpx}, cuPDQPx uses
a primal weight to balance the primal and dual updates. We parameterize the step sizes as
\begin{equation}
\label{eq:steps}
\tau=\frac{\eta}{\omega},\qquad \sigma=\frac{\eta\omega}{\delta},\qquad \beta=\eta\,\delta\,\omega ,
\end{equation}
where $\eta>0$ is the step size, $\omega>0$ is the primal weight, and $\delta>0$
balances the auxiliary dual block $w$ against the constraint dual block $y$.

Under this parameterization, the step-size condition~\eqref{eq:stepcond} becomes
\[
\eta^2\norm{\delta Q+\delta^{-1}A^\top A}_2\le1.
\]
This condition is independent of $\omega$. We call $\delta$ the quadratic weight because
it balances the contributions of $Q$ and $A^\top A$ in this condition.
We choose the step size as
\begin{equation}
\label{eq:eta}
\eta=\frac{\gamma_\eta}{\sqrt{\big\|\delta Q+\delta^{-1}A^\top A\big\|_2}},
\end{equation}
where $\gamma_\eta$ is a fixed step-size factor (we set it to be $0.99$ in the solver package) and the operator norm in the denominator of \eqref{eq:eta} is estimated by power
iteration.

The weights are initialized and updated at restarts as described in Section~\ref{sec:weights}.
The step size is recomputed whenever $\delta$ is updated.

\subsection{Adaptive restart}
\label{sec:restart}

Following the adaptive restart strategy of cuPDLPx \citep{lu2025cupdlpx,lu2024halpern},
we use the fixed-point error to decide
when to restart the reflected Halpern iteration~\eqref{eq:halpern}:
\begin{equation}
\label{eq:fperr}
r_{\mathrm{fp}}(z)=\norm{z-T(z)}_{P},
\end{equation}
where $P$ is the norm matrix in~\eqref{eq:P}.
The parameters and hence $P$ remain fixed within each restart window.

The squared fixed-point error under the parameterization~\eqref{eq:steps} is
\[
r_{\mathrm{fp}}(z)^2=\frac{\omega}{\eta}\norm{\Delta x}_2^2
+\frac{1}{\eta\delta\omega}\ip{\Delta w}{Q\Delta w}
+\frac{\delta}{\eta\omega}\norm{\Delta y}_2^2
-2\ip{Q\Delta w}{\Delta x}+2\ip{A^\top\Delta y}{\Delta x},
\]
where $(\Delta x,\Delta w,\Delta y)=z-T(z)$.

Let $z^{n,k}$ denote the $k$th iterate in restart window $n$, with anchor $z^{n,0}$,
and let $K$ be the total iteration count. We check the restart conditions every $H$
iterations, so the previous check is at $z^{n,k-H}$. A restart is triggered when any of
the following conditions holds:
\begin{align}
r_{\mathrm{fp}}(z^{n,k})&\le\beta_1 r_{\mathrm{fp}}(z^{n,0}),
\tag{sufficient decay}\label{eq:restart}\\[3pt]
r_{\mathrm{fp}}(z^{n,k})&\le\beta_2 r_{\mathrm{fp}}(z^{n,0})
\quad\text{and}\quad r_{\mathrm{fp}}(z^{n,k})>r_{\mathrm{fp}}(z^{n,k-H}),
\tag{necessary decay + no local progress}\label{eq:restart-necessary}\\[3pt]
k&\ge\beta_3 K
\tag{artificial restart}\label{eq:restart-artificial}
\end{align}
At a restart, we set both the anchor and the initial iterate of the next window to the
non-reflected PDHG point $T(z^{n,k})$ and reset the Halpern counter.
We then update the weights $\omega,\delta$ and the step size as described in
Sections~\ref{sec:weights} and~\ref{sec:stepselection}.

\subsection{Primal and quadratic weights}
\label{sec:weights}

Following the primal-weight strategy of PDLP \citep{applegate2021pdlp}, we adapt the weights to
balance the relative scales of the variable updates. Under the step-size parameterization~\eqref{eq:steps},
the primal weight $\omega$ balances the primal block $x$ against the dual blocks $(w,y)$,
while the quadratic weight $\delta$ balances the auxiliary quadratic block $w$ against the
constraint dual block $y$.

We use the movements over each restart window to determine target weights, accounting also for the dependence of the step size $\eta$ on $\delta$.
Both weights remain fixed within each window and are updated at restarts.

\paragraph{Initialization.}
For $A\ne0$ and $Q\ne0$, we initialize the quadratic weight by solving
\begin{equation}
\label{eq:delta0}
\delta_0\in\argmin_{\delta>0}\norm{\delta Q+\delta^{-1}A^\top A}_2.
\end{equation}
We approximate the minimizer using golden-section search, with
$\norm{A}_2/\sqrt{\norm{Q}_2}$ used to initialize the search.
The operator norms are estimated by power iteration.
If $A=0$ or $Q=0$, we set $\delta_0=1$.

The initial primal weight is
\begin{equation}
\label{eq:omega0}
\omega_0=\frac{\norm{c}_2}{\norm{\max\{|\ell_c|,|u_c|\}}_2}.
\end{equation}
The maximum is componentwise, with infinite bound entries treated as zero.
All norms are evaluated on the rescaled problem data.
We replace the ratio in~\eqref{eq:omega0} by $1$ if it exceeds $10^3$, before clamping
$\omega_0$ to $[10^{-3},10^3]$.

\paragraph{Target weights.}
Let $\Delta x$, $\Delta w$, and $\Delta y$ denote the movements of the non-reflected PDHG iterates
over the completed restart window, with $\norm{\Delta w}_Q^2=\ip{\Delta w}{Q\Delta w}$.
The metric in~\eqref{eq:P} motivates the separable surrogate
\begin{equation}
\label{eq:sepmetric}
\Phi(\omega,\delta)\ =\ \omega\norm{\Delta x}^2+\frac{1}{\delta\omega}\norm{\Delta w}_Q^2+\frac{\delta}{\omega}\norm{\Delta y}^2 .
\end{equation}
Minimizing $\Phi$ over $\omega$ with $\delta$ fixed gives
\begin{equation}
\label{eq:omegastar}
\omega^\star=\frac{\sqrt{\delta^{-1}\norm{\Delta w}_Q^2+\delta\norm{\Delta y}^2}}{\norm{\Delta x}}.
\end{equation}
For the quadratic weight, we also account for the dependence of the step size on $\delta$.
Using the separable norm bound, we minimize $\Phi(\omega,\delta)/\eta(\delta)$ with
$\eta(\delta)\propto(\delta\norm{Q}+\delta^{-1}\norm{A}^2)^{-1/2}$. Eliminating $\omega$ gives
\begin{equation}
\label{eq:deltastar}
\delta^\star=\sqrt{\frac{\norm{A}_2}{\sqrt{\norm{Q}_2}}\cdot\frac{\norm{\Delta w}_Q}{\norm{\Delta y}}}.
\end{equation}
Thus the quadratic-weight target combines the norm ratio $\norm{A}_2/\sqrt{\norm{Q}_2}$
with the observed movement ratio.

\paragraph{Adaptive update rule.}
At restart $n$, we first damp the target~\eqref{eq:deltastar} on a logarithmic scale with
factor $\gamma_\delta$, then limit the multiplicative change by $\kappa_\delta$:
\begin{equation}
\label{eq:deltaupdate}
\delta_{n+1}=\proj_{[\delta_n/\kappa_\delta,\,\kappa_\delta\delta_n]}
\bigl(\delta_n^{1-\gamma_\delta}(\delta_n^\star)^{\gamma_\delta}\bigr).
\end{equation}
We leave $\delta$ unchanged when $\norm{\Delta y}$ or $\norm{\Delta w}_Q$ is negligible relative
to $\norm{\Delta x}$. We then evaluate the primal-weight target~\eqref{eq:omegastar} at the new
$\delta$ and apply the PID controller of cuPDLPx \citep{lu2025cupdlpx}.
With the integral and derivative gains set to zero ($K_I=K_D=0$), the update becomes
\begin{equation}
\label{eq:omegaupdate}
e_n=\log\frac{\omega^\star_n}{\omega_n},\qquad
\log\omega_{n+1}=\log\omega_n+K_P e_n,
\end{equation}
where $K_P$ is the proportional gain.

When $\delta$ changes, we recompute $\eta$ from~\eqref{eq:eta} using one power method and update
$\beta=\eta\delta\omega$ in~\eqref{eq:iter-w}.

\subsection{Termination criteria}
\label{sec:termination}

cuPDQPx terminates with status \textsc{optimal} when the relative primal residual,
relative dual residual, and relative primal--dual gap are all below the tolerance $\eps$.
These measures are defined with respect to the primal problem~\eqref{eq:qp} and its
Wolfe-type dual~\eqref{eq:dual}. At a primal--dual iterate $(x,y)$ with bound multiplier $r$,
they are given by
\begin{gather}
\frac{\norm{Ax-\proj_S(Ax)}_\infty}
{1+\max\{\norm{(\ell_c,u_c)}_\infty,\norm{Ax}_\infty\}}
\tag{primal feasibility}\label{eq:term-primal}\\[3pt]
\frac{\norm{Qx+c-A^\top y-r}_\infty}
{1+\max\{\norm{c}_\infty,\norm{A^\top y}_\infty,\norm{Qx}_\infty\}}
\tag{dual feasibility}\label{eq:term-dual}\\[3pt]
\frac{\big|x^\top Qx+c^\top x+p(-y;\ell_c,u_c)+p(-r;\ell_v,u_v)\big|}
{1+\max\big\{\big|\tfrac12x^\top Qx+c^\top x\big|,
\big|\tfrac12x^\top Qx+p(-y;\ell_c,u_c)+p(-r;\ell_v,u_v)\big|\big\}}
\tag{primal--dual gap}\label{eq:term-gap}
\end{gather}
where infinite bound entries are treated as zero when computing $\norm{(\ell_c,u_c)}_\infty$.
We evaluate all three measures at the non-reflected PDHG point using the original, unscaled
problem data.

In the notation of~\eqref{eq:iter}, we recover the bound multiplier via the Moreau identity
applied to the primal update~\eqref{eq:iter-x}:
\[
r^{k+1}=\tau^{-1}(x^{k+1}-x^k)+c+Qw^k-A^\top y^k.
\]
The multiplier is zero wherever the projection is inactive.

Termination is checked periodically, with additional checks during the initial iterations.
These additional checks leave the restart schedule in Section~\ref{sec:restart} unchanged.
If the time limit is reached, cuPDQPx returns the latest non-reflected PDHG point with time-limit
status. The check frequencies and experimental time limit are specified in
Sections~\ref{sec:impl} and~\ref{sec:setup}, respectively.

\subsection{Algorithm summary}

Algorithm~\ref{sec:procedure} summarizes cuPDQPx. The indices $n$ and $k$ denote restart windows
and iterations within a window. The point $\tilde{z}^{n,k+1}=T(z^{n,k})$ is the non-reflected PDHG
update, and $z^{n,k+1}$ is the Halpern iterate. The map $T$ uses the current weights and step size.
Returned points are mapped to the original variables. Termination and restart tests are performed
only at the scheduled checks described in Sections~\ref{sec:termination} and~\ref{sec:restart}.

\begingroup
\setlength{\parfillskip}{0pt plus 1fil}
\begin{algorithm}[H]
\caption{Overview of the core algorithm implemented in cuPDQPx}
\label{sec:procedure}
\DontPrintSemicolon
\small
\KwIn{QP~\eqref{eq:qp}, initial point $z^{0,0}$, tolerance $\eps$.}
Apply diagonal preconditioning. \tcp*{Section~\ref{sec:precond}}
Initialize the step size $\eta$ and weights $\omega,\delta$.
    \tcp*{Sections~\ref{sec:stepselection} and~\ref{sec:weights}}
\For{$n=0,1,\ldots$}{
    $k\gets0$.\;
    \Repeat(\tcp*[f]{Adaptive restart; Section~\ref{sec:restart}}){a restart condition holds}{
        $\tilde{z}^{n,k+1}\gets T(z^{n,k})$.
            \tcp*{PDHG update; Section~\ref{sec:scheme}}
        \If(\tcp*[f]{Termination check; Section~\ref{sec:termination}}){the termination criteria
        hold at $\tilde{z}^{n,k+1}$}{
            \Return{$(\tilde{x}^{n,k+1},\tilde{y}^{n,k+1})$.}\;
        }
        $z^{n,k+1}\gets\frac{k+1}{k+2}(2T(z^{n,k})-z^{n,k})
            +\frac{1}{k+2}z^{n,0}$.
            \tcp*{Halpern and reflection; Section~\ref{sec:halpern}}
        $k\gets k+1$.\;
    }
    $z^{n+1,0}\gets\tilde{z}^{n,k}$. \tcp*{Reset the restart anchor}
    Update $\delta,\omega$ and the step size $\eta$. \tcp*{Section~\ref{sec:weights}}
}
\end{algorithm}
\endgroup

\section{Diagonal quadratic objectives}
\label{app:diagonal}

The auxiliary-variable scheme in Sections~\ref{sec:framework}--\ref{sec:enhancements}
yields explicit updates for nondiagonal $Q$ without matrix factorizations or inner quadratic
solves. When $Q$ is diagonal, the primal proximal subproblem~\eqref{eq:proxsub} separates
across coordinates and admits a closed-form solution. This structure arises in practical
QPs and is common in our benchmark sets (Table~\ref{tab:sets}).
We therefore apply PDHG directly to~\eqref{eq:saddle}, together with the reflected Halpern
framework and enhancements of cuPDLPx \citep{lu2025cupdlpx}.
This specialization requires no auxiliary variable or quadratic-weight adaptation, handles
$Q$ through elementwise operations, and has a step-size condition depending only on $A$.
These simplifications can make it more efficient than the auxiliary-variable scheme on
diagonal QPs.

The resulting two-block structure also simplifies the practical enhancements. We reuse the
preconditioning, adaptive restart, and primal-weight update strategies developed for LP in
PDLP \citep{applegate2021pdlp} and cuPDLPx \citep{lu2025cupdlpx}, with the adjustments
described below. The PDHCG and HPR-QP-C implementations also provide specialized handling
of diagonal QPs (Table~\ref{tab:baseline-implementation}), although this specialization is
not explicitly described in their papers \citep{chen2025hprqp,huang2025pdhcg}.

\subsection{Primal--dual updates and convergence}
\label{sec:diagonal-updates}

For diagonal $Q$, we write $Q=\mathrm{Diag}(q)$, where $q_i=Q_{ii}$ is the $i$th diagonal entry.
Applying PDHG~\eqref{eq:pdhg-prox} gives the explicit updates
\begin{subequations}
\label{eq:diag-iter}
\begin{align}
x^{k+1} &= \proj_X\!\left((I+\tau Q)^{-1}\big(x^k+\tau(A^\top y^k-c)\big)\right),
\label{eq:diagprox}\\
y^{k+1} &= y^k-\sigma\Big(A(2x^{k+1}-x^k)
-\proj_{S}\big(A(2x^{k+1}-x^k)-\sigma^{-1}y^k\big)\Big). \label{eq:diag-dual}
\end{align}
\end{subequations}
Here $z^k=(x^k,y^k)$ and $z^{k+1}=T(z^k)$ denotes one PDHG step defined
by~\eqref{eq:diag-iter}.
The optimality conditions of~\eqref{eq:diag-iter} give a proximal point representation
for the saddle-point problem~\eqref{eq:saddle}, with
\begin{equation}
\label{eq:P-diag}
P_{\mathrm{diag}}(\tau,\sigma)=
\begin{pmatrix}
\frac{1}{\tau}I & A^\top\\[2pt]
A & \frac{1}{\sigma}I
\end{pmatrix}.
\end{equation}
For $\tau,\sigma>0$, this matrix is positive semidefinite if and only if
\begin{equation}
\label{eq:diag-stepcond}
\tau\sigma\norm{A}_2^2\le1.
\end{equation}
Under this condition, the monotonicity argument in Proposition~\ref{prop:psd}, applied
to the optimality operator of~\eqref{eq:saddle}, shows that $T$ is firmly nonexpansive
in $\norm{\cdot}_{P_{\mathrm{diag}}}$. Its fixed points are exactly the saddle points
of~\eqref{eq:saddle}. Step-size selection is described in Section~\ref{sec:diagonal-enhancements}.

\paragraph{Reflected Halpern scheme.}
Assume $Q\succeq0$ and that~\eqref{eq:qp} has a KKT point. Let $T$ be defined
by~\eqref{eq:diag-iter} with fixed $\tau,\sigma>0$ satisfying~\eqref{eq:diag-stepcond},
and set $P=P_{\mathrm{diag}}(\tau,\sigma)$ as in~\eqref{eq:P-diag}.
Then $T$ and $2T-I$ are nonexpansive in $\norm{\cdot}_P$, and the fixed points of $T$
coincide with the KKT points of~\eqref{eq:qp} in the $(x,y)$ components.
The iterates generated by~\eqref{eq:halpern} with a fixed anchor satisfy the
bound~\eqref{eq:rate} for every fixed point $z^\star$ of $T$.
Indeed, the proximal point argument above gives firm nonexpansiveness of $T$ in
$\norm{\cdot}_P$, and the quotient-space argument in the proof of
Theorem~\ref{thm:rate} applies with $P=P_{\mathrm{diag}}$.
All theoretical guarantees of cuPDLPx~\citep{lu2024halpern,lu2025cupdlpx} carry over in this case.

\subsection{Algorithmic enhancements}
\label{sec:diagonal-enhancements}

\paragraph{Step size.}
For diagonal $Q$, the quadratic term is included in the closed-form primal
update~\eqref{eq:diagprox}, and condition~\eqref{eq:diag-stepcond} depends only on $A$.
We set $\delta\equiv1$ and choose
\begin{equation}
\eta=\frac{\gamma_\eta}{\norm{A}_2},
\end{equation}
where $\norm{A}_2$ is also estimated by power iteration.

\paragraph{Adaptive restart.}
We use the restart conditions in Section~\ref{sec:restart}, with $P=P_{\mathrm{diag}}$
in~\eqref{eq:fperr}. Under the parameterization~\eqref{eq:steps} with $\delta\equiv1$,
we have
\[
r_{\mathrm{fp}}(z)^2=\frac{\omega}{\eta}\norm{\Delta x}_2^2
+\frac{1}{\eta\omega}\norm{\Delta y}_2^2
+2\ip{A^\top\Delta y}{\Delta x},
\]
where $(\Delta x,\Delta y)=z-T(z)$.

\paragraph{Primal weight.}
We fix $\delta\equiv1$ and update only the primal weight, initialized as
\begin{equation}
\label{eq:omega0-diag}
\omega_0=\frac{\norm{c}_2}{\norm{\max\{|\ell_c|,|u_c|\}}_2}
+\frac{\norm{Q}_2}{\norm{A}_2}.
\end{equation}
We use the norm and bound conventions and the initial primal-weight safeguard described in
Section~\ref{sec:weights}.

Let $\Delta x$ and $\Delta y$ denote the movements of the non-reflected PDHG iterates over
one restart window, as in Section~\ref{sec:weights}. Since the auxiliary block is absent,
the separable surrogate reduces to $\omega\norm{\Delta x}^2+\omega^{-1}\norm{\Delta y}^2$.
Minimizing it with respect to $\omega$ gives the primal-weight target
\begin{equation}
\label{eq:omegadiag}
\omega^\star=\frac{\norm{\Delta y}}{\norm{\Delta x}} .
\end{equation}

We use~\eqref{eq:omegaupdate} and skip windows without primal or dual
movement.

We suspend primal-weight updates when repeated moves in the same
direction do not sufficiently reduce the residuals. Let $\rho_n$ be the larger of the relative
primal and dual residuals in
Section~\ref{sec:termination} at restart $n$, and let $d_n$ be the proposed logarithmic change
$\log(\omega_{n+1}/\omega_n)$. If
\begin{equation}
\label{eq:brake}
d_n d_{n-1}>0 \quad\text{and}\quad \rho_n>\theta_b\,\rho_{n-1}
\end{equation}
holds at $N_b$ consecutive restarts, the move is set to zero until the condition clears.
Here $\theta_b$ is the residual reduction threshold.

Algorithm~\ref{sec:procedure} applies with $z=(x,y)$ and the map $T$ in~\eqref{eq:diag-iter},
using the step size, fixed-point error, and primal-weight updates specified above.

\subsection{Termination and implementation}
\label{sec:diagonal-implementation}

\paragraph{Termination criteria.}
We use the termination criteria in Section~\ref{sec:termination}. In the notation
of~\eqref{eq:diag-iter}, the bound multiplier is recovered via the Moreau identity:
\[
r^{k+1}=\tau^{-1}(x^{k+1}-x^k)+c+Qx^{k+1}-A^\top y^k.
\]
The multiplier is zero wherever the projection is inactive.

\paragraph{GPU implementation.}
cuPDQPx automatically selects the specialized scheme~\eqref{eq:diag-iter} when it detects
diagonal $Q$. We store the diagonal as a vector and evaluate the quadratic terms using
elementwise vector operations instead of SpMVs. Combined with the closed-form primal
update~\eqref{eq:diagprox}, the only remaining SpMVs involve $A$ and $A^\top$.
Thus, each iteration between termination checks requires two SpMVs and no global reduction.

\paragraph{Default parameters.}
We use $\alpha=0.5$ in Pock--Chambolle scaling and $K_P=0.5$ in the primal-weight update.
The weight-update suspension rule uses $(\theta_b,N_b)=(0.9,3)$.
The remaining parameters and check frequencies are as specified in Section~\ref{sec:impl}.

\section{Numerical Experiments}
\label{sec:experiments}

This section evaluates the numerical performance of cuPDQPx on the Maros--M\'esz\'aros
collection, a convex subset of QPLIB, and synthetic QPs.
We first describe the GPU implementation and default parameters, then present the
experimental setup and comparisons of solved counts and running times.
Additional comparisons by quadratic structure and iteration count are provided in
Appendix~\ref{sec:additional-results}.
Our aim is to provide empirical evidence of practical performance under the stated solver
configurations, rather than to establish general superiority of the underlying algorithm
over any other approaches.

\subsection{Implementation details and default parameters}
\label{sec:impl}

cuPDQPx is implemented in CUDA C and follows the GPU implementation design of cuPDLPx
\citep{lu2025cupdlpx}. More specifically, we incorporate several GPU-specific optimizations, including optimized sparse matrix-vector products (SpMVs), kernel fusion, and CUDA Graphs.
For SpMVs, the problem matrix $A$ and its explicit transpose $A^\top$ are stored in compressed sparse row (CSR) format. We use the cuSPARSE backend for SpMVs, which automatically selects between SpMV and SpMVOp based on the CUDA version, preferring SpMVOp when supported.
We additionally use cuSPARSE preprocessing to reduce the overhead of repeated SpMVs, while all working vectors remain resident in GPU memory.
For kernel fusion, we combine several vector operations that appear within each iteration, including projections, reflections, and Halpern averaging, into fused CUDA kernels. This reduces both kernel-launch overhead and unnecessary accesses to global memory.
Finally, we use CUDA Graphs to reduce the launch overhead associated with repeatedly executing the same sequence of kernels. Between consecutive termination checks, a batch of iterations is captured and executed as a CUDA graph.

\paragraph{SpMV fusion for the quadratic term.}
Compared with cuPDLPx, solving QPs introduces additional SpMVs involving the quadratic matrix $Q$.
In the PDHG step~\eqref{eq:iter}, both $A$ and $Q$ are applied to the same extrapolated point
$2x^{k+1}-x^k$, allowing the two products to be fused into a single SpMV using the vertically
stacked matrix $[A;Q]$.
Multiplying~\eqref{eq:iter-w} by $Q$ allows the implementation to maintain $Qw$ directly
using the already computed $Q(2x^{k+1}-x^k)$, so no additional SpMV with $Q$ is needed.
Together with the SpMV for $A^\top y$ in the primal update, each iteration between termination
checks therefore requires only two SpMVs and no global reduction.

\paragraph{Default parameters.}
Table~\ref{tab:defaults} summarizes the default parameters for the auxiliary-variable scheme
by algorithmic component. Section~\ref{app:diagonal} gives the settings specific to diagonal $Q$.

\begin{table}[htb]
\centering
\caption{Default parameter values for the auxiliary-variable scheme in cuPDQPx.}
\label{tab:defaults}
\small
\begin{tabular}{@{}llp{0.38\linewidth}@{}}
\toprule
Component & Parameters & Default values\\
\midrule
Ruiz equilibration & & $10$ iterations\\
$L_2$ rescaling & & $1$ pass\\
Pock--Chambolle scaling & & $1$ pass ($\alpha=1$)\\
Step-size factor & $\gamma_\eta$ & $0.99$\\
Adaptive restart & $(\beta_1,\beta_2,\beta_3)$ & $(0.2,0.5,0.2)$\\
Quadratic-weight update & $(\gamma_\delta,\kappa_\delta)$
 & $(0.1,1.2)$\\
Primal-weight update & $K_P$ & $1$\\
Termination check frequency & $H$ & $100$\\
\bottomrule
\end{tabular}
\end{table}

The experiments use the specified parameters, with regular termination and restart checks every
$H=100$ iterations. Additional termination checks are performed every
$10$ iterations during the first $1000$ iterations. These early checks can change the reported
termination count but leave the trajectory unchanged, since restarts remain on the
$100$-iteration grid.

\subsection{Experimental setup}
\label{sec:setup}

\paragraph{Benchmark sets.}
We evaluate cuPDQPx on three benchmark sets: the Maros--M\'esz\'aros collection
\citep{maros1999repository}, a convex subset of QPLIB \citep{furini2019qplib},
and synthetic QPs \citep{stellato2020osqp}.
Table~\ref{tab:sets} summarizes the number of instances, the distribution of
diagonal and nondiagonal $Q$, and the ranges of the numbers of variables and
constraints in each benchmark set.
For Maros--M\'esz\'aros and QPLIB, we measure problem size by $n+m$, following
\citep{lu2026pdqp}. Table~\ref{tab:size-groups} defines the small, medium, and large groups
and reports their instance counts.
The synthetic set contains seven problem classes and three size levels, with three instances
per class at each size level, so that the small, medium, and large groups each contain
21 instances. These size levels follow
\citep{lu2026pdqp} and are defined within each problem class, separately from the $n+m$
groups used for Maros--M\'esz\'aros and QPLIB.

\suppressfloats[t]
\begin{table}[htb]
\centering
\caption{Benchmark sets used in the experiments.}
\label{tab:sets}
\small
\setlength{\tabcolsep}{4pt}
\begin{tabular}{lcccrr}
\toprule
Set & Instances & Diagonal $Q$ & Nondiagonal $Q$ &
\multicolumn{1}{c}{$n$ range} & \multicolumn{1}{c}{$m$ range}\\
\midrule
Maros--M\'esz\'aros & 134 & 55 & 79 & 2\,--\,93,263 & 1\,--\,186,531\\
QPLIB, convex subset & 33 & 22 & 11 & 150\,--\,1,009,306 & 5\,--\,1,001,000\\
Synthetic & 63 & 36 & 27 & 502\,--\,424,711 & 81\,--\,143,922\\
\bottomrule
\end{tabular}
\par\smallskip
\begin{minipage}{\linewidth}
\footnotesize
Here, $n$ is the number of variables and $m$ is the number of linear constraints, excluding variable bounds.
\end{minipage}
\end{table}

\begin{table}[htb]
\centering
\caption{Size groups and instance counts for the Maros--M\'esz\'aros and QPLIB sets.}
\label{tab:size-groups}
\small
\setlength{\tabcolsep}{8pt}
\begin{tabular}{lccc}
\toprule
\multirow{2}{*}{Set} & Small & Medium & Large\\
 & $n+m<10^3$ & $10^3\le n+m<5\times10^4$ & $n+m\ge5\times10^4$\\
\midrule
Maros--M\'esz\'aros & 60 & 69 & 5\\
QPLIB, convex subset & 5 & 19 & 9\\
\bottomrule
\end{tabular}
\par\smallskip
\begin{minipage}{\linewidth}
\footnotesize
The size groups use the dimensions of the original problem, before presolve, with variable
bounds excluded from the constraint count. The synthetic size levels are defined separately
within each problem class.
\end{minipage}
\end{table}

\paragraph{Baselines.}
We compare cuPDQPx with HPR-QP-C\insert\footins{\footnotesize
\noindent All solvers were accessed in September 2026.\par
}\footnote{HPR-QP-C: \texttt{v0.1.0}, \url{https://github.com/PolyU-IOR/HPR-QP-C}.}
\citep{chen2025hprqp},
PDHCG\footnote{PDHCG: \texttt{v0.3.1},
\url{https://github.com/Lhongpei/PDHCG}.} \citep{huang2025pdhcg},
PDQP.jl\footnote{PDQP.jl: commit \texttt{135035a},
\url{https://github.com/jinwen-yang/PDQP.jl}.} \citep{lu2026pdqp},
Gurobi barrier\footnote{Gurobi: \texttt{v13.0.3},
\url{https://www.gurobi.com}.} \citep{gurobi2026},
Clarabel\footnote{Clarabel: \texttt{v0.11.1},
\url{https://github.com/oxfordcontrol/Clarabel.jl}.} \citep{goulart2026clarabel},
CuClarabel\footnote{CuClarabel: \texttt{CuClarabel} branch,
commit \texttt{50dff3ec}, \url{https://github.com/oxfordcontrol/Clarabel.jl}.}
\citep{chen2026cuclarabel},
and SCS\footnote{SCS: \texttt{v3.3.1},
\url{https://github.com/cvxgrp/scs}.} \citep{odonoghue2016scs}.
We benchmark Gurobi and Clarabel on the CPU; CuClarabel, HPR-QP-C, PDHCG, PDQP.jl, and cuPDQPx on the GPU; and SCS on both the CPU and GPU.
For Gurobi, we use the barrier method without crossover
(\texttt{Method=2}, \texttt{Crossover=0}), with presolve disabled.
For Clarabel, we use MKL Pardiso on the CPU, while CuClarabel uses cuDSS on the GPU.
For SCS, the CPU configuration uses MKL Pardiso, while the GPU configuration uses cuDSS with 32-bit integers.
Presolve is disabled for PDHCG and for cuPDQPx in the main solver comparisons.
We further evaluate the effect of presolve on cuPDQPx using PSLP\footnote{PSLP:
\texttt{v0.3.0},
\url{https://github.com/dance858/PSLP}.}.
Apart from these settings and the termination-related parameters described below,
all other solver parameters are left at their defaults.
The Gurobi, Clarabel, CuClarabel, and SCS runs use JuMP 1.31.2 and
MathOptInterface (MOI) 1.54.0, with Julia 1.12.5 except for CuClarabel
(Julia 1.11.9).

Among these baselines, HPR-QP-C, PDHCG, PDQP.jl, and cuPDQPx are factorization-free first-order solvers implemented for GPUs. We therefore compare their implementation characteristics in Table~\ref{tab:baseline-implementation}, including diagonal-$Q$ handling, SpMV backends, and regular termination-check intervals.

\suppressfloats[t]
\begin{table}[htb]
\centering
\caption{Implementation characteristics of the compared first-order solvers.}
\label{tab:baseline-implementation}
\small
\begin{tabular}{@{}p{0.11\linewidth}p{0.11\linewidth}p{0.18\linewidth}p{0.33\linewidth}
                  p{\dimexpr0.27\linewidth-8\tabcolsep\relax}@{}}
\toprule
Solver & Language & Separate\newline diagonal-$Q$ path & SpMV backend & \raggedright Check interval\newline
(iterations)\tabularnewline
\midrule
cuPDQPx & CUDA C & Yes
 & \raggedright cuSPARSE SpMV/SpMVOp & $100$\tabularnewline
\addlinespace
HPR-QP-C & CUDA C & Yes
 & \raggedright cuSPARSE / custom SpMV & $100$\tabularnewline
\addlinespace
PDHCG & CUDA C & Yes
 & \raggedright cuSPARSE SpMV/SpMVOp & $200$\tabularnewline
\addlinespace
PDQP.jl & Julia & No
 & \raggedright cuSPARSE SpMV & $96$\tabularnewline
\bottomrule
\end{tabular}
\par\smallskip
\begin{minipage}{\linewidth}
\footnotesize
In HPR-QP-C's default automatic mode, timing trials select the SpMV backend separately for $Q$ and
for the pair $A$, $A^\top$. Intervals refer to regular termination checks;
cuPDQPx and HPR-QP-C also check every $10$ iterations during the first $1000$ iterations.
PDQP.jl additionally checks during the initial iterations.
\end{minipage}
\end{table}

\paragraph{Computing environment.}
All experiments were conducted on a Linux system with a single NVIDIA H100-SXM-80GB GPU
(CUDA 13.3) and Intel Xeon Platinum 8462Y+ processors (16 CPU cores).

\paragraph{Termination criteria.}
We report results at tolerances $\eps\in\{10^{-4},10^{-8}\}$ and impose a
time limit of $3600$ seconds on all runs.
cuPDQPx uses the relative $\ell_\infty$ criteria described in Section~\ref{sec:termination},
while the other solvers use their native termination criteria. The regular termination-check
intervals of the factorization-free first-order solvers are summarized in
Table~\ref{tab:baseline-implementation}.
For cuPDQPx and PDHCG, we set \texttt{eps\_opt} and \texttt{eps\_feas} to $\eps$.
For PDQP.jl, we set \texttt{eps\_\allowbreak optimal\_\allowbreak absolute} and
\texttt{eps\_\allowbreak optimal\_\allowbreak relative} to $\eps$.
For Gurobi, we set \texttt{BarConvTol}, \texttt{FeasibilityTol}, and
\texttt{OptimalityTol} to $\eps$. For Clarabel and CuClarabel, we set
\texttt{tol\_gap\_abs}, \texttt{tol\_gap\_rel}, and \texttt{tol\_feas} to $\eps$.
For SCS, we set \texttt{eps\_abs} and \texttt{eps\_rel} to $\eps$ and remove the iteration limit.
HPR-QP-C is run once per instance with \texttt{stoptol} set to $10^{-8}$.
For each looser tolerance, we report the iteration count and elapsed time at the first
termination check satisfying that tolerance.
An instance is counted as solved if the prescribed termination criterion is satisfied within
the time limit. For solvers other than HPR-QP-C, we require an optimal termination status;
approximate solutions, suboptimal terminations, infeasibility reports, and execution failures
are counted as unsuccessful.
Because the solvers use different stopping criteria, the same tolerance value does not
necessarily imply identical termination accuracy.

\paragraph{Performance measures.}
\label{sec:cost}
We report solved counts and summarize running times using the shifted geometric mean with shift $10$,
$\mathrm{SGM}_{10}=\exp\big(\frac1N\sum_i\ln(v_i+10)\big)-10$.
For running times, $v_i$ is measured in seconds, and each mean includes all instances in the
corresponding group. A run that does not attain the reporting tolerance within the
time limit is assigned $3600$\,s.
For Gurobi, Clarabel, CuClarabel, and SCS, the reported time is MOI's solve time and excludes
file reading. We use this value directly, without adding the separately recorded read time.

We also report $\mathrm{SGM}_{10}$ iteration counts for cuPDQPx, HPR-QP-C, PDHCG, and PDQP.jl.
For each benchmark set and tolerance, the \emph{common solved set} contains the instances
solved within the time limit by all four solvers. All solver means within each group use
this same set. For PDHCG, we count outer iterations.
Reported speedups are ratios of the baseline's $\mathrm{SGM}_{10}$ running time to that of
cuPDQPx, computed from the unrounded values.

\subsection{Benchmark results}
\label{sec:comparison}

Tables~\ref{tab:main},~\ref{tab:qplib-runtime}, and~\ref{tab:synthetic-runtime} compare
the solvers on the Maros--M\'esz\'aros, QPLIB, and synthetic sets, respectively, with
results grouped by problem size as defined in Section~\ref{sec:setup}.
Table~\ref{tab:cupdqpx-presolve} compares cuPDQPx with and without PSLP presolve.
Figure~\ref{fig:performance-profiles} shows the performance profiles
of the solvers on all three benchmark sets.
Appendix~\ref{sec:additional-results} provides additional computational results.

In Tables~\ref{tab:main}--\ref{tab:synthetic-runtime}, Count denotes the number of
solved instances, and Time denotes
the $\mathrm{SGM}_{10}$ running time in seconds. Group sizes appear in parentheses.
Within each group and tolerance, bold indicates the highest solved count or lowest
running time, with all count ties highlighted.

\begin{table}[htbp]
\centering
\caption{Solved instances and running times in seconds on the Maros--M\'esz\'aros set,
at nominal tolerances $\eps$ with a reporting time limit of $3600$\,s.}
\label{tab:main}
\footnotesize
\setlength{\tabcolsep}{3pt}
\begin{tabular}{llrrrrrrrr}
\toprule
\multirow{2}{*}{$\eps$} & \multirow{2}{*}{Solver}
 & \multicolumn{2}{c}{Small (60)} & \multicolumn{2}{c}{Medium (69)}
 & \multicolumn{2}{c}{Large (5)} & \multicolumn{2}{c}{Total (134)}\\
\cmidrule(lr){3-4}\cmidrule(lr){5-6}\cmidrule(lr){7-8}\cmidrule(lr){9-10}
 & & Count & Time & Count & Time & Count & Time & Count & Time\\
\midrule
\multirow{9}{*}{$10^{-4}$} & Gurobi & \textbf{60} & \textbf{0.004}
 & \textbf{69} & \textbf{0.17} & \textbf{5} & \textbf{1.28}
 & \textbf{134} & \textbf{0.14}\\
 & Clarabel & \textbf{60} & 0.009
 & 68 & 1.09 & 4 & 34.09
 & 132 & 1.15\\
 & CuClarabel & \textbf{60} & 0.07
 & 67 & 2.06 & 4 & 25.70
 & 131 & 1.58\\
\cmidrule(lr){2-10}
 & SCS (CPU) & \textbf{60} & 0.03
 & \textbf{69} & 1.48 & 4 & 34.87
 & 133 & 1.37\\
 & SCS (GPU) & \textbf{60} & 0.27
 & \textbf{69} & 0.95 & 4 & 29.59
 & 133 & 1.16\\
 & PDQP.jl & 59 & 6.45
 & 66 & 11.93 & \textbf{5} & 105.63
 & 130 & 10.51\\
 & PDHCG & \textbf{60} & 1.27
 & 62 & 14.74 & \textbf{5} & 35.25
 & 127 & 7.80\\
 & HPR-QP-C & \textbf{60} & 0.18
 & 67 & 12.20 & \textbf{5} & 11.19
 & 132 & 5.63\\
 & cuPDQPx & \textbf{60} & 0.10
 & \textbf{69} & 9.09 & \textbf{5} & 5.29
 & \textbf{134} & 4.24\\
\midrule
\multirow{9}{*}{$10^{-8}$} & Gurobi & \textbf{60} & \textbf{0.004}
 & \textbf{67} & \textbf{2.08} & \textbf{5} & \textbf{1.15}
 & \textbf{132} & \textbf{1.07}\\
 & Clarabel & \textbf{60} & 0.01
 & 59 & 13.96 & 3 & 125.28
 & 122 & 7.29\\
 & CuClarabel & \textbf{60} & 0.08
 & 59 & 13.90 & 3 & 103.53
 & 122 & 7.21\\
\cmidrule(lr){2-10}
 & SCS (CPU) & \textbf{60} & 0.07
 & 63 & 10.32 & 4 & 40.13
 & 127 & 5.34\\
 & SCS (GPU) & \textbf{60} & 0.62
 & 62 & 12.02 & 4 & 36.23
 & 126 & 6.33\\
 & PDQP.jl & 59 & 13.61
 & 57 & 70.62 & \textbf{5} & 695.60
 & 121 & 40.44\\
 & PDHCG & \textbf{60} & 4.09
 & 60 & 26.51 & \textbf{5} & 64.47
 & 125 & 14.48\\
 & HPR-QP-C & \textbf{60} & 0.81
 & 61 & 21.20 & \textbf{5} & 29.29
 & 126 & 9.57\\
 & cuPDQPx & \textbf{60} & 0.27
 & 62 & 16.33 & \textbf{5} & 16.71
 & 127 & 7.28\\
\bottomrule
\end{tabular}
\end{table}

\begin{table}[htbp]
\centering
\caption{Solved instances and running times in seconds on the QPLIB set,
at nominal tolerances $\eps$ with a reporting time limit of $3600$\,s.}
\label{tab:qplib-runtime}
\footnotesize
\setlength{\tabcolsep}{3pt}
\begin{tabular}{llrrrrrrrr}
\toprule
\multirow{2}{*}{$\eps$} & \multirow{2}{*}{Solver}
 & \multicolumn{2}{c}{Small (5)} & \multicolumn{2}{c}{Medium (19)}
 & \multicolumn{2}{c}{Large (9)} & \multicolumn{2}{c}{Total (33)}\\
\cmidrule(lr){3-4}\cmidrule(lr){5-6}\cmidrule(lr){7-8}\cmidrule(lr){9-10}
 & & Count & Time & Count & Time & Count & Time & Count & Time\\
\midrule
\multirow{9}{*}{$10^{-4}$} & Gurobi & \textbf{5} & 0.02
 & \textbf{19} & \textbf{0.41} & \textbf{9} & \textbf{12.05}
 & \textbf{33} & \textbf{2.70}\\
 & Clarabel & \textbf{5} & 0.04
 & 17 & 9.12 & 5 & 191.86
 & 27 & 22.98\\
 & CuClarabel & \textbf{5} & 0.08
 & 18 & 3.94 & 7 & 35.70
 & 30 & 8.35\\
\cmidrule(lr){2-10}
 & SCS (CPU) & \textbf{5} & \textbf{0.02}
 & \textbf{19} & 0.79 & 7 & 71.52
 & 31 & 8.53\\
 & SCS (GPU) & \textbf{5} & 0.07
 & \textbf{19} & 0.68 & \textbf{9} & 39.52
 & \textbf{33} & 6.08\\
 & PDQP.jl & \textbf{5} & 0.25
 & 18 & 13.61 & 6 & 218.76
 & 29 & 28.65\\
 & PDHCG & \textbf{5} & 0.19
 & 18 & 6.28 & 8 & 71.60
 & 31 & 13.54\\
 & HPR-QP-C & \textbf{5} & 0.03
 & \textbf{19} & 2.12 & \textbf{9} & 24.39
 & \textbf{33} & 5.65\\
 & cuPDQPx & \textbf{5} & 0.03
 & \textbf{19} & 1.39 & 8 & 45.02
 & 32 & 7.17\\
\midrule
\multirow{9}{*}{$10^{-8}$} & Gurobi & \textbf{5} & \textbf{0.02}
 & \textbf{19} & \textbf{0.51} & \textbf{9} & \textbf{13.10}
 & \textbf{33} & \textbf{2.93}\\
 & Clarabel & \textbf{5} & 0.09
 & 17 & 9.37 & 5 & 205.87
 & 27 & 23.86\\
 & CuClarabel & \textbf{5} & 0.11
 & 17 & 9.04 & 5 & 153.47
 & 27 & 21.09\\
\cmidrule(lr){2-10}
 & SCS (CPU) & \textbf{5} & 0.04
 & 17 & 10.88 & 4 & 290.36
 & 26 & 28.66\\
 & SCS (GPU) & \textbf{5} & 0.15
 & 17 & 12.26 & 5 & 156.78
 & 27 & 24.23\\
 & PDQP.jl & \textbf{5} & 22.59
 & 18 & 33.47 & 4 & 566.99
 & 27 & 74.24\\
 & PDHCG & 4 & 23.46
 & 18 & 10.73 & 7 & 107.20
 & 29 & 25.75\\
 & HPR-QP-C & \textbf{5} & 1.10
 & 18 & 7.53 & 7 & 156.88
 & 30 & 20.24\\
 & cuPDQPx & \textbf{5} & 0.96
 & 18 & 6.09 & 7 & 97.49
 & 30 & 15.49\\
\bottomrule
\end{tabular}
\end{table}

\begin{table}[htbp]
\centering
\captionsetup{skip=8pt}
\caption{Solved instances and running times in seconds on the synthetic set,
at nominal tolerances $\eps$ with a reporting time limit of $3600$\,s.}
\label{tab:synthetic-runtime}
\setlength{\parskip}{0pt}
\footnotesize
\setlength{\tabcolsep}{3pt}
\begin{tabular}{llrrrrrrrr}
\toprule
\multirow{2}{*}{$\eps$} & \multirow{2}{*}{Solver}
 & \multicolumn{2}{c}{Small (21)} & \multicolumn{2}{c}{Medium (21)}
 & \multicolumn{2}{c}{Large (21)} & \multicolumn{2}{c}{Total (63)}\\
\cmidrule(lr){3-4}\cmidrule(lr){5-6}\cmidrule(lr){7-8}\cmidrule(lr){9-10}
 & & Count & Time & Count & Time & Count & Time & Count & Time\\
\midrule
\multirow{9}{*}{$10^{-4}$} & Gurobi & \textbf{21} & 0.48
 & \textbf{21} & 5.33 & \textbf{21} & 56.04
 & \textbf{63} & 11.98\\
 & Clarabel & \textbf{21} & 0.48
 & \textbf{21} & 3.89 & \textbf{21} & 39.51
 & \textbf{63} & 9.32\\
 & CuClarabel & \textbf{21} & 0.22
 & \textbf{21} & 1.14 & \textbf{21} & 11.76
 & \textbf{63} & 3.53\\
\cmidrule(lr){2-10}
 & SCS (CPU) & \textbf{21} & 0.35
 & \textbf{21} & 4.05 & \textbf{21} & 110.95
 & \textbf{63} & 16.01\\
 & SCS (GPU) & \textbf{21} & 0.27
 & \textbf{21} & 1.50 & \textbf{21} & 13.33
 & \textbf{63} & 4.02\\
 & PDQP.jl & \textbf{21} & 0.83
 & \textbf{21} & 1.77 & \textbf{21} & 5.32
 & \textbf{63} & 2.50\\
 & PDHCG & \textbf{21} & 0.29
 & \textbf{21} & 0.68 & \textbf{21} & 3.13
 & \textbf{63} & 1.30\\
 & HPR-QP-C & \textbf{21} & 0.08
 & \textbf{21} & 0.19 & \textbf{21} & 1.18
 & \textbf{63} & 0.47\\
 & cuPDQPx & \textbf{21} & \textbf{0.03}
 & \textbf{21} & \textbf{0.10} & \textbf{21} & \textbf{1.10}
 & \textbf{63} & \textbf{0.40}\\
\midrule
\multirow{9}{*}{$10^{-8}$} & Gurobi & \textbf{21} & 0.51
 & \textbf{21} & 5.43 & \textbf{21} & 56.32
 & \textbf{63} & 12.07\\
 & Clarabel & \textbf{21} & 0.56
 & \textbf{21} & 4.64 & \textbf{21} & 45.17
 & \textbf{63} & 10.43\\
 & CuClarabel & \textbf{21} & 0.25
 & \textbf{21} & 1.28 & \textbf{21} & 12.34
 & \textbf{63} & 3.72\\
\cmidrule(lr){2-10}
 & SCS (CPU) & \textbf{21} & 0.54
 & \textbf{21} & 8.59 & \textbf{21} & 159.67
 & \textbf{63} & 22.15\\
 & SCS (GPU) & \textbf{21} & 0.57
 & \textbf{21} & 3.73 & \textbf{21} & 30.08
 & \textbf{63} & 7.99\\
 & PDQP.jl & \textbf{21} & 2.26
 & \textbf{21} & 4.65 & \textbf{21} & 10.97
 & \textbf{63} & 5.56\\
 & PDHCG & \textbf{21} & 0.62
 & \textbf{21} & 1.45 & 20 & 12.91
 & 62 & 4.07\\
 & HPR-QP-C & \textbf{21} & 0.24
 & \textbf{21} & 0.65 & \textbf{21} & 4.14
 & \textbf{63} & 1.56\\
 & cuPDQPx & \textbf{21} & \textbf{0.07}
 & \textbf{21} & \textbf{0.29} & \textbf{21} & \textbf{2.55}
 & \textbf{63} & \textbf{0.91}\\
\bottomrule
\end{tabular}
\end{table}

\begin{figure}[p]
\centering
\includegraphics[width=\linewidth]{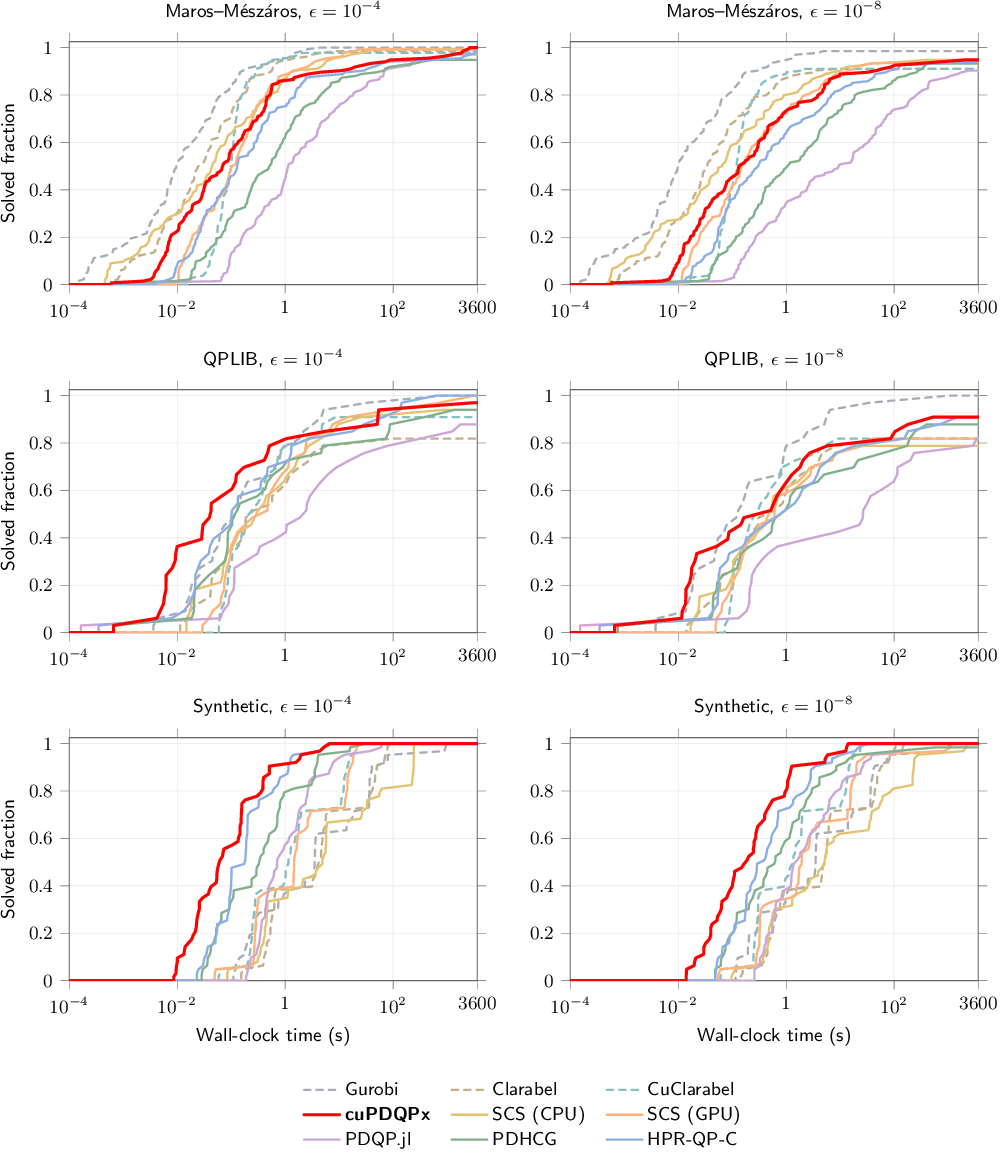}
\caption{Cumulative fractions of solved instances versus wall-clock time at $\eps=10^{-4}$ (left) and
$\eps=10^{-8}$ (right), with a time limit of $3600$\,s.
Each row corresponds to one benchmark set. Horizontal axes are logarithmic and use the
same range in all panels.}
\label{fig:performance-profiles}
\end{figure}

\paragraph{Maros--M\'esz\'aros.}
Gurobi achieves the lowest overall $\mathrm{SGM}_{10}$ running time at both tolerances (Table~\ref{tab:main}). The open-source factorization-based solvers Clarabel and SCS also demonstrate competitive performance, particularly at $\eps=10^{-4}$. This is consistent with the relatively small size of the Maros--M\'esz\'aros instances, for which sparse matrix factorizations are generally inexpensive. Among the four factorization-free GPU solvers, cuPDQPx achieves the lowest overall running time and highest solved count, solving all $134$ instances at $\eps=10^{-4}$ and $127$ at $\eps=10^{-8}$.

\paragraph{QPLIB.}
Gurobi continues to achieve the lowest overall running time, solving all $33$ instances at both tolerances (Table~\ref{tab:qplib-runtime}). Among the open-source solvers, HPR-QP-C and SCS (GPU) achieve lower overall running times than cuPDQPx at $\eps=10^{-4}$. At $\eps=10^{-8}$, however, cuPDQPx achieves the lowest overall running time among all open-source solvers, solving $30$ instances with an $\mathrm{SGM}_{10}$ time of $15.49$\,s.

\paragraph{Synthetic QPs.}
In contrast to the standard benchmark sets, cuPDQPx achieves the lowest $\mathrm{SGM}_{10}$ running time among all tested solvers on the synthetic set, solving all $63$ instances at both tolerances (Table~\ref{tab:synthetic-runtime}). Its advantage is particularly pronounced on larger instances. At $\eps=10^{-8}$, cuPDQPx achieves an $\mathrm{SGM}_{10}$ time of $2.55$\,s on the large group, compared with $56.32$\,s for Gurobi, $12.34$\,s for CuClarabel, and $30.08$\,s for SCS (GPU). Although GPU acceleration substantially improves the performance of Clarabel and SCS relative to their CPU implementations, cuPDQPx remains considerably faster on these instances.

\paragraph{Effect of solution accuracy.}
The solvers exhibit different sensitivities to the termination tolerance. When tightening $\eps$ from $10^{-4}$ to $10^{-8}$, Gurobi's overall $\mathrm{SGM}_{10}$ running time remains nearly unchanged on QPLIB and the synthetic set, consistent with the fast local convergence of interior-point methods. Clarabel exhibits a more pronounced increase on Maros--M\'esz\'aros but relatively little change on the other two sets. In contrast, SCS and the factorization-free GPU solvers generally exhibit larger increases in running time on QPLIB and the synthetic set as the tolerance becomes tighter.

\paragraph{Effect of presolve.}
Table~\ref{tab:cupdqpx-presolve} compares cuPDQPx with and without PSLP presolve at both
tolerances. Total time includes presolve, whereas cuPDQPx time excludes it.
The two configurations solve the same instances in every benchmark set.
On Maros--M\'esz\'aros, total running times are nearly unchanged.
On QPLIB, presolve reduces the total $\mathrm{SGM}_{10}$ time from $7.17$\,s to $6.23$\,s
at $\eps=10^{-4}$, with a smaller reduction at $\eps=10^{-8}$.
On the synthetic set, PSLP leaves most instances unchanged, and the reduction in cuPDQPx
time is too small to offset the additional presolve cost.

\begin{table}[htbp]
\centering
\caption{Solved instances and running times in seconds of cuPDQPx with and without PSLP presolve,
with a reporting time limit of $3600$\,s.}
\label{tab:cupdqpx-presolve}
\small
\setlength{\tabcolsep}{6pt}
\begin{tabular}{@{}lclrrr@{}}
\toprule
Set & $\eps$ & Presolve & Count
 & \multicolumn{1}{c}{Total time} & \multicolumn{1}{c}{cuPDQPx time}\\
\midrule
\multirow{4}{*}[-1.5pt]{Maros--M\'esz\'aros (134)} & \multirow{2}{*}{$10^{-4}$}
 & Disabled & 134 & 4.24 & 4.24\\
 & & Enabled & 134 & 4.23 & 4.23\\
\addlinespace[3pt]
 & \multirow{2}{*}{$10^{-8}$} & Disabled & 127 & 7.28 & 7.28\\
 & & Enabled & 127 & 7.23 & 7.23\\
\midrule
\multirow{4}{*}[-1.5pt]{QPLIB (33)} & \multirow{2}{*}{$10^{-4}$}
 & Disabled & 32 & 7.17 & 7.17\\
 & & Enabled & 32 & 6.23 & 6.20\\
\addlinespace[3pt]
 & \multirow{2}{*}{$10^{-8}$} & Disabled & 30 & 15.49 & 15.49\\
 & & Enabled & 30 & 15.20 & 15.19\\
\midrule
\multirow{4}{*}[-1.5pt]{Synthetic (63)} & \multirow{2}{*}{$10^{-4}$}
 & Disabled & 63 & 0.40 & 0.40\\
 & & Enabled & 63 & 0.85 & 0.37\\
\addlinespace[3pt]
 & \multirow{2}{*}{$10^{-8}$} & Disabled & 63 & 0.91 & 0.91\\
 & & Enabled & 63 & 1.28 & 0.84\\
\bottomrule
\end{tabular}
\end{table}

\paragraph{Additional experimental results.}
Appendix~\ref{sec:additional-results} provides detailed comparisons of cuPDQPx with the
factorization-free GPU solvers PDQP.jl, HPR-QP-C, and PDHCG, with separate results for
diagonal and nondiagonal QPs.
In addition to solved counts and running times, it reports iteration counts on the
common solved sets (Tables~\ref{tab:mm-work},~\ref{tab:qplib-work}, and~\ref{tab:synthetic-work})
and per-iteration computational work measured by SpMV counts (Table~\ref{tab:cost}).
On the common solved sets, cuPDQPx achieves the lowest overall $\mathrm{SGM}_{10}$ iteration
count among these four solvers in five of the six benchmark--tolerance
combinations. Its iteration advantage is more consistent for nondiagonal $Q$, where it has
the lowest mean across all three benchmark sets at both tolerances.
On the nondiagonal subsets of Maros--M\'esz\'aros and QPLIB, cuPDQPx also achieves lower
$\mathrm{SGM}_{10}$ running times than the other three solvers at both tolerances.

\section{Conclusion}
\label{sec:conclusion}

We presented cuPDQPx, a GPU-accelerated, factorization-free solver for large-scale convex QP based on restarted reflected Halpern PDHG. An auxiliary-variable splitting yields explicit primal--dual updates for nondiagonal $Q$ without linearizing the quadratic objective or solving inner quadratic subproblems, while a closed-form proximal update handles diagonal $Q$. We established a preconditioned proximal point interpretation and an $O(1/k)$ fixed-point residual bound under fixed parameters and a fixed anchor. The solver further combines adaptive restarts, primal and quadratic weight updates, and efficient GPU implementation techniques.

Numerical experiments on Maros--M\'esz\'aros, QPLIB, and synthetic QPs demonstrate the competitive performance of cuPDQPx. While Gurobi barrier achieves the lowest overall running times on the standard benchmark sets, cuPDQPx outperforms the other factorization-free GPU solvers on the nondiagonal-$Q$ subsets of both sets. On the synthetic set, cuPDQPx achieves the lowest shifted geometric mean running time among all tested solvers at both tolerances, highlighting the potential of factorization-free GPU methods for large-scale QP.

Our experiments focus on feasible, bounded QPs, and cuPDQPx does not yet support infeasibility detection. Future work will investigate more robust adaptive weight updates, particularly under stagnation, and extend the solver to detect infeasibility.

\section*{Acknowledgement}
Haihao Lu is partially supported by AFOSR Grant No. FA9550-24-1-0051, ONR Grant No. N000142412735, and a Sloan Research Fellowship.
Zedong Peng is partially supported by ONR Grant No. N000142412735. Jinwen Yang is partially supported by AFOSR Grant No. FA9550-24-1-0051. The authors gratefully acknowledge the support of the NVIDIA Academic Grant Program.

\bibliographystyle{plainnat}
\bibliography{refs}

@article{lu2025cupdlpjl,
  title   = {{cuPDLP}.jl: A {GPU} implementation of restarted primal-dual hybrid gradient for linear programming in {Julia}},
  author  = {Lu, Haihao and Yang, Jinwen},
  journal = {Operations Research},
  volume  = {73},
  number  = {6},
  pages   = {3440--3452},
  year    = {2025},
  doi     = {10.1287/opre.2024.1069}
}

@article{lu2023cupdlpc,
  title={{cuPDLP-C}: A Strengthened Implementation of {cuPDLP} for Linear Programming by {C} language},
  author={Lu, Haihao and Yang, Jinwen and Hu, Haodong and Huangfu, Qi and Liu, Jinsong and Liu, Tianhao and Ye, Yinyu and Zhang, Chuwen and Ge, Dongdong},
  journal={arXiv preprint arXiv:2312.14832},
  year={2023}
}

@article{applegate2021pdlp,
  title   = {Practical large-scale linear programming using primal-dual hybrid gradient},
  author  = {Applegate, David and D{\'\i}az, Mateo and Hinder, Oliver and Lu, Haihao and Lubin, Miles and O'Donoghue, Brendan and Schudy, Warren},
  journal = {Advances in Neural Information Processing Systems},
  volume  = {34},
  pages   = {20243--20257},
  year    = {2021}
}

@article{chambolle2011first,
  title   = {A first-order primal-dual algorithm for convex problems with applications to imaging},
  author  = {Chambolle, Antonin and Pock, Thomas},
  journal = {Journal of Mathematical Imaging and Vision},
  volume  = {40},
  pages   = {120--145},
  year    = {2011}
}

@incollection{robinson1981polyhedral,
  author    = {Robinson, Stephen M.},
  title     = {Some Continuity Properties of Polyhedral Multifunctions},
  booktitle = {Mathematical Programming at Oberwolfach},
  series    = {Mathematical Programming Studies},
  volume    = {14},
  pages     = {206--214},
  year      = {1981},
  publisher = {Springer},
  doi       = {10.1007/BFb0120929}
}

@inproceedings{pock2011diagonal,
  title     = {Diagonal preconditioning for first order primal-dual algorithms in convex optimization},
  author    = {Pock, Thomas and Chambolle, Antonin},
  booktitle = {2011 International Conference on Computer Vision},
  pages     = {1762--1769},
  year      = {2011}
}

@article{lu2026pdqp,
  title   = {A practical and optimal first-order method for large-scale convex quadratic programming},
  author  = {Lu, Haihao and Yang, Jinwen},
  journal = {Mathematical Programming},
  volume  = {215},
  number  = {1-2},
  pages   = {771--808},
  year    = {2026},
  doi     = {10.1007/s10107-025-02241-0}
}

@article{lu2023unified,
  title   = {On a unified and simplified proof for the ergodic convergence rates of
             {PPM}, {PDHG} and {ADMM}},
  author  = {Lu, Haihao and Yang, Jinwen},
  journal = {arXiv preprint arXiv:2305.02165},
  year    = {2023}
}

@article{lu2024halpern,
  title   = {Restarted {Halpern} {PDHG} for linear programming},
  author  = {Lu, Haihao and Yang, Jinwen},
  journal = {arXiv preprint arXiv:2407.16144},
  year    = {2024}
}

@article{lu2025cupdlpx,
  title   = {{cuPDLPx}: A further enhanced {GPU}-based first-order solver for linear programming},
  author  = {Lu, Haihao and Peng, Zedong and Yang, Jinwen},
  journal = {arXiv preprint arXiv:2507.14051},
  year    = {2025}
}

@article{chen2025hprqp,
  title   = {{HPR-QP}: A dual {Halpern} {Peaceman--Rachford} method for solving large-scale convex composite quadratic programming},
  author  = {Chen, Kaihuang and Sun, Defeng and Yuan, Yancheng and Zhang, Guojun and Zhao, Xinyuan},
  journal = {arXiv preprint arXiv:2507.02470},
  year    = {2025}
}

@article{chen2026hprlp,
  title   = {{HPR-LP}: An implementation of an {HPR} method for solving linear programming},
  author  = {Chen, Kaihuang and Sun, Defeng and Yuan, Yancheng and Zhang, Guojun and Zhao, Xinyuan},
  journal = {Mathematical Programming Computation},
  volume  = {18},
  number  = {1},
  pages   = {183--210},
  year    = {2026},
  doi     = {10.1007/s12532-025-00292-0}
}

@article{huang2025pdhcg,
  title   = {A restarted primal-dual hybrid conjugate gradient method for large-scale quadratic programming},
  author  = {Huang, Yicheng and Zhang, Wanyu and Li, Hongpei and Ge, Dongdong and Liu, Huikang and Ye, Yinyu},
  journal = {INFORMS Journal on Computing},
  year    = {2025},
  doi     = {10.1287/ijoc.2024.0983}
}

@article{li2026pdhcgcqp,
  title   = {{GPU}-accelerated conic quadratic programming with local linear convergence under strict complementarity},
  author  = {Li, Hongpei and Huang, Yicheng and Liu, Huikang and Ge, Dongdong and Ye, Yinyu},
  journal = {arXiv preprint arXiv:2608.09159},
  year    = {2026}
}

@article{lieder2021halpern,
  title   = {On the convergence rate of the {Halpern}-iteration},
  author  = {Lieder, Felix},
  journal = {Optimization Letters},
  volume  = {15},
  number  = {2},
  pages   = {405--418},
  year    = {2021}
}

@article{maros1999repository,
  title   = {A repository of convex quadratic programming problems},
  author  = {Maros, Istv{\'a}n and M{\'e}sz{\'a}ros, Csaba},
  journal = {Optimization Methods and Software},
  volume  = {11},
  number  = {1-4},
  pages   = {671--681},
  year    = {1999}
}

@article{furini2019qplib,
  title   = {{QPLIB}: a library of quadratic programming instances},
  author  = {Furini, Fabio and Traversi, Emiliano and Belotti, Pietro and Frangioni, Antonio and Gleixner, Ambros and Gould, Nick and Liberti, Leo and Lodi, Andrea and Misener, Ruth and Mittelmann, Hans and Sahinidis, Nikolaos V. and Vigerske, Stefan and Wiegele, Angelika},
  journal = {Mathematical Programming Computation},
  volume  = {11},
  number  = {2},
  pages   = {237--265},
  year    = {2019}
}

@article{wolfe1961duality,
  title   = {A duality theorem for non-linear programming},
  author  = {Wolfe, Philip},
  journal = {Quarterly of Applied Mathematics},
  volume  = {19},
  number  = {3},
  pages   = {239--244},
  year    = {1961}
}

@article{stellato2020osqp,
  title   = {{OSQP}: an operator splitting solver for quadratic programs},
  author  = {Stellato, Bartolomeo and Banjac, Goran and Goulart, Paul and Bemporad, Alberto and Boyd, Stephen},
  journal = {Mathematical Programming Computation},
  volume  = {12},
  number  = {4},
  pages   = {637--672},
  year    = {2020}
}

@misc{gurobi2026,
  author = {{Gurobi Optimization, LLC}},
  title  = {{Gurobi Optimizer Reference Manual}},
  year   = {2026},
  url    = {https://www.gurobi.com}
}

@article{goulart2026clarabel,
  title   = {{Clarabel}: An interior-point solver for conic programs with quadratic objectives},
  author  = {Goulart, Paul J. and Chen, Yuwen},
  journal = {Mathematical Programming Computation},
  year    = {2026},
  doi     = {10.1007/s12532-026-00320-7}
}

@article{chen2026cuclarabel,
  title   = {{CuClarabel}: {GPU} acceleration for a conic optimization solver},
  author  = {Chen, Yuwen and Tse, Danny and Nobel, Parth and Goulart, Paul and Boyd, Stephen},
  journal = {ACM Transactions on Mathematical Software},
  volume  = {52},
  number  = {3},
  pages   = {1--25},
  year    = {2026},
  doi     = {10.1145/3815420}
}

@article{odonoghue2016scs,
  title   = {Conic Optimization via Operator Splitting and Homogeneous Self-Dual Embedding},
  author  = {O'Donoghue, Brendan and Chu, Eric and Parikh, Neal and Boyd, Stephen},
  journal = {Journal of Optimization Theory and Applications},
  volume  = {169},
  number  = {3},
  pages   = {1042--1068},
  year    = {2016}
}

\clearpage
\appendix

\clearpage
\section{Additional computational results}
\label{sec:additional-results}

Tables~\ref{tab:mm-runtime-structure} and~\ref{tab:qplib-runtime-structure} report solved
counts and running times for Maros--M\'esz\'aros and QPLIB, grouped by the structure of $Q$.
They use the same runs and reporting conventions as the size-grouped results in
Section~\ref{sec:comparison}.

\begin{table}[htbp]
\centering
\caption{Solved instances and running times on the Maros--M\'esz\'aros set, grouped by the structure of $Q$,
at nominal tolerances $\eps$ with a reporting time limit of $3600$\,s.}
\label{tab:mm-runtime-structure}
\footnotesize
\setlength{\tabcolsep}{3pt}
\begin{tabular}{llrrrrrr}
\toprule
\multirow{2}{*}{$\eps$} & \multirow{2}{*}{Solver}
 & \multicolumn{2}{c}{Diagonal $Q$ (55)} & \multicolumn{2}{c}{Nondiagonal $Q$ (79)} & \multicolumn{2}{c}{Total (134)}\\
\cmidrule(lr){3-4}\cmidrule(lr){5-6}\cmidrule(lr){7-8}
 & & Count & Time & Count & Time & Count & Time\\
\midrule
\multirow{9}{*}{$10^{-4}$} & Gurobi & \textbf{55} & \textbf{0.21} & \textbf{79} & 0.09
 & \textbf{134} & \textbf{0.14}\\
 & Clarabel & 53 & 2.91 & \textbf{79} & \textbf{0.07} & 132 & 1.15\\
 & CuClarabel & 52 & 4.10 & \textbf{79} & 0.10 & 131 & 1.58\\
\cmidrule(lr){2-8}
 & SCS (CPU) & 54 & 3.39 & \textbf{79} & 0.14 & 133 & 1.37\\
 & SCS (GPU) & 54 & 2.48 & \textbf{79} & 0.33 & 133 & 1.16\\
 & PDQP.jl & 52 & 17.04 & 78 & 6.92 & 130 & 10.51\\
 & PDHCG & 48 & 21.51 & \textbf{79} & 1.95 & 127 & 7.80\\
 & HPR-QP-C & 53 & 17.76 & \textbf{79} & 0.48 & 132 & 5.63\\
 & cuPDQPx & \textbf{55} & 13.00 & \textbf{79} & 0.20 & \textbf{134} & 4.24\\
\midrule
\multirow{9}{*}{$10^{-8}$} & Gurobi & \textbf{54} & \textbf{1.36} & 78 & 0.87
 & \textbf{132} & \textbf{1.07}\\
 & Clarabel & 44 & 23.71 & 78 & 0.87 & 122 & 7.29\\
 & CuClarabel & 44 & 23.13 & 78 & 0.91 & 122 & 7.21\\
\cmidrule(lr){2-8}
 & SCS (CPU) & 48 & 16.13 & \textbf{79} & 0.59 & 127 & 5.34\\
 & SCS (GPU) & 47 & 16.76 & \textbf{79} & 1.58 & 126 & 6.33\\
 & PDQP.jl & 43 & 102.97 & 78 & 18.78 & 121 & 40.44\\
 & PDHCG & 46 & 33.00 & \textbf{79} & 6.54 & 125 & 14.48\\
 & HPR-QP-C & 47 & 31.99 & \textbf{79} & 1.51 & 126 & 9.57\\
 & cuPDQPx & 48 & 25.03 & \textbf{79} & \textbf{0.57} & 127 & 7.28\\
\bottomrule
\end{tabular}
\end{table}

\begin{table}[htbp]
\centering
\caption{Solved instances and running times on the QPLIB set, grouped by the structure of $Q$,
at nominal tolerances $\eps$ with a reporting time limit of $3600$\,s.}
\label{tab:qplib-runtime-structure}
\footnotesize
\setlength{\tabcolsep}{3pt}
\begin{tabular}{llrrrrrr}
\toprule
\multirow{2}{*}{$\eps$} & \multirow{2}{*}{Solver}
 & \multicolumn{2}{c}{Diagonal $Q$ (22)} & \multicolumn{2}{c}{Nondiagonal $Q$ (11)} & \multicolumn{2}{c}{Total (33)}\\
\cmidrule(lr){3-4}\cmidrule(lr){5-6}\cmidrule(lr){7-8}
 & & Count & Time & Count & Time & Count & Time\\
\midrule
\multirow{9}{*}{$10^{-4}$} & Gurobi & \textbf{22} & \textbf{3.91} & \textbf{11} & 0.59
 & \textbf{33} & \textbf{2.70}\\
 & Clarabel & 16 & 49.09 & \textbf{11} & 0.28 & 27 & 22.98\\
 & CuClarabel & 20 & 8.86 & 10 & 7.36 & 30 & 8.35\\
\cmidrule(lr){2-8}
 & SCS (CPU) & 20 & 14.55 & \textbf{11} & 0.55 & 31 & 8.53\\
 & SCS (GPU) & \textbf{22} & 9.96 & \textbf{11} & 0.45 & \textbf{33} & 6.08\\
 & PDQP.jl & 18 & 51.66 & \textbf{11} & 5.19 & 29 & 28.65\\
 & PDHCG & 20 & 24.29 & \textbf{11} & 1.08 & 31 & 13.54\\
 & HPR-QP-C & \textbf{22} & 9.10 & \textbf{11} & 0.52 & \textbf{33} & 5.65\\
 & cuPDQPx & 21 & 12.26 & \textbf{11} & \textbf{0.21} & 32 & 7.17\\
\midrule
\multirow{9}{*}{$10^{-8}$} & Gurobi & \textbf{22} & \textbf{4.20} & \textbf{11} & 0.73
 & \textbf{33} & \textbf{2.93}\\
 & Clarabel & 16 & 51.12 & \textbf{11} & \textbf{0.39} & 27 & 23.86\\
 & CuClarabel & 17 & 31.54 & 10 & 7.41 & 27 & 21.09\\
\cmidrule(lr){2-8}
 & SCS (CPU) & 16 & 43.32 & 10 & 10.33 & 26 & 28.66\\
 & SCS (GPU) & 17 & 31.88 & 10 & 12.86 & 27 & 24.23\\
 & PDQP.jl & 16 & 104.93 & \textbf{11} & 35.25 & 27 & 74.24\\
 & PDHCG & 19 & 31.56 & 10 & 16.45 & 29 & 25.75\\
 & HPR-QP-C & 19 & 37.85 & \textbf{11} & 2.08 & 30 & 20.24\\
 & cuPDQPx & 19 & 28.85 & \textbf{11} & 0.97 & 30 & 15.49\\
\bottomrule
\end{tabular}
\end{table}

Table~\ref{tab:cost} summarizes the SpMVs per iteration, grouped by the structure of $Q$.

\begin{table}[htbp]
\centering
\caption{Sparse matrix--vector products per iteration.}
\label{tab:cost}
\small
\begin{tabular}{llccccc}
\toprule
$Q$ structure & Solver & $A$ & $A^\top$ & $Q$ & $[A;Q]$ & Total\\
\midrule
\multirow{4}{*}{Nondiagonal}
 & cuPDQPx & 0 & 1 & 0 & 1 & 2\\
 & HPR-QP-C & 1 & 1 & 2 & 0 & 4\\
 & PDQP.jl & 1 & 1 & 1 & 0 & 3\\
 & PDHCG & 1 & 1 & $j$ & 0 & $j+2$\\
\midrule
\multirow{4}{*}{Diagonal}
 & cuPDQPx & 1 & 1 & 0 & 0 & 2\\
 & HPR-QP-C & 1 & 1 & 2 & 0 & 4\\
 & PDQP.jl & 1 & 1 & 1 & 0 & 3\\
 & PDHCG & 1 & 1 & 0 & 0 & 2\\
\bottomrule
\end{tabular}
\par\smallskip
\begin{minipage}{\linewidth}
\footnotesize
Counts refer to cuSPARSE SpMV/SpMVOp calls or custom SpMV kernel launches. They exclude
vector-only kernels, termination checks, restarts, and elementwise products with diagonal $Q$.
For PDHCG, $j$ is the number of inner solver steps per outer iteration, including the initial
step.
HPR-QP-C counts refer to the standard branch; its alternative branch, used when more than
$80\%$ of variables have neither a finite lower nor a finite upper bound, requires one fewer
product with $Q$.
\end{minipage}
\end{table}

Tables~\ref{tab:mm-work},~\ref{tab:qplib-work}, and~\ref{tab:synthetic-work} report iteration
counts on the common solved sets defined in Section~\ref{sec:cost}.
At $\eps=10^{-4}$, the common sets contain $123$, $29$, and $63$ instances for
Maros--M\'esz\'aros, QPLIB, and Synthetic, respectively;
at $\eps=10^{-8}$, they contain $121$, $26$, and $62$ instances.
These iteration comparisons complement the solved counts and running times in
Tables~\ref{tab:main},~\ref{tab:qplib-runtime}, and~\ref{tab:synthetic-runtime}.

\begin{table}[htbp]
\centering
\caption{Iteration counts on the Maros--M\'esz\'aros set,
reported as $\mathrm{SGM}_{10}$ over instances solved by all four solvers within $3600$\,s.}
\label{tab:mm-work}
\small
\setlength{\tabcolsep}{4pt}
\begin{tabular}{llrrr}
\toprule
$\eps$ & Solver & Diagonal $Q$ & Nondiagonal $Q$ & Total\\
\midrule
\multirow{5}{*}{$10^{-4}$} & \textcolor{black!75}{Common}
 & \textcolor{black!75}{$N=45$} & \textcolor{black!75}{$N=78$}
 & \textcolor{black!75}{$N=123$}\\
\addlinespace[3pt]
 & PDQP.jl & 6891 & 5531 & 5995\\
 & PDHCG & 8430 & 2567 & 3968\\
 & HPR-QP-C & 6948 & 2536 & 3668\\
 & cuPDQPx & \textbf{5691} & \textbf{2308} & \textbf{3212}\\
\midrule
\multirow{5}{*}{$10^{-8}$} & \textcolor{black!75}{Common}
 & \textcolor{black!75}{$N=43$} & \textcolor{black!75}{$N=78$}
 & \textcolor{black!75}{$N=121$}\\
\addlinespace[3pt]
 & PDQP.jl & 23269 & 17296 & 19219\\
 & PDHCG & 19936 & 7718 & 10815\\
 & HPR-QP-C & 17111 & 7523 & 10075\\
 & cuPDQPx & \textbf{14340} & \textbf{6321} & \textbf{8458}\\
\bottomrule
\end{tabular}
\par\smallskip
\begin{minipage}{\linewidth}
\footnotesize
$N$ is the number of common solved instances at the indicated tolerance.
Zero-iteration solves are included.
Bold indicates the lowest mean for each group and tolerance.
\end{minipage}
\end{table}

\begin{table}[htbp]
\centering
\caption{Iteration counts on the QPLIB set,
reported as $\mathrm{SGM}_{10}$ over instances solved by all four solvers within $3600$\,s.}
\label{tab:qplib-work}
\small
\setlength{\tabcolsep}{4pt}
\begin{tabular}{llrrr}
\toprule
$\eps$ & Solver & Diagonal $Q$ & Nondiagonal $Q$ & Total\\
\midrule
\multirow{5}{*}{$10^{-4}$} & \textcolor{black!75}{Common}
 & \textcolor{black!75}{$N=18$} & \textcolor{black!75}{$N=11$}
 & \textcolor{black!75}{$N=29$}\\
\addlinespace[3pt]
 & PDQP.jl & 8439 & 2244 & 5109\\
 & PDHCG & 6495 & 1843 & 4030\\
 & HPR-QP-C & \textbf{5842} & 1711 & \textbf{3669}\\
 & cuPDQPx & 6353 & \textbf{1552} & 3725\\
\midrule
\multirow{5}{*}{$10^{-8}$} & \textcolor{black!75}{Common}
 & \textcolor{black!75}{$N=16$} & \textcolor{black!75}{$N=10$}
 & \textcolor{black!75}{$N=26$}\\
\addlinespace[3pt]
 & PDQP.jl & 17588 & 9098 & 13650\\
 & PDHCG & 10287 & 8258 & 9453\\
 & HPR-QP-C & \textbf{8084} & 6274 & 7333\\
 & cuPDQPx & 8777 & \textbf{4946} & \textbf{7040}\\
\bottomrule
\end{tabular}
\par\smallskip
\begin{minipage}{\linewidth}
\footnotesize
$N$ is the number of common solved instances at the indicated tolerance.
Zero-iteration solves are included.
Bold indicates the lowest mean for each group and tolerance.
\end{minipage}
\end{table}

\begin{table}[htbp]
\centering
\caption{Iteration counts on the synthetic set,
reported as $\mathrm{SGM}_{10}$ over instances solved by all four solvers within $3600$\,s.}
\label{tab:synthetic-work}
\small
\setlength{\tabcolsep}{4pt}
\begin{tabular}{llrrr}
\toprule
$\eps$ & Solver & Diagonal $Q$ & Nondiagonal $Q$ & Total\\
\midrule
\multirow{5}{*}{$10^{-4}$} & \textcolor{black!75}{Common}
 & \textcolor{black!75}{$N=36$} & \textcolor{black!75}{$N=27$}
 & \textcolor{black!75}{$N=63$}\\
\addlinespace[3pt]
 & PDQP.jl & 1443 & 7549 & 2936\\
 & PDHCG & 954 & 1611 & 1194\\
 & HPR-QP-C & 933 & 1603 & 1177\\
 & cuPDQPx & \textbf{645} & \textbf{1270} & \textbf{863}\\
\midrule
\multirow{5}{*}{$10^{-8}$} & \textcolor{black!75}{Common}
 & \textcolor{black!75}{$N=36$} & \textcolor{black!75}{$N=26$}
 & \textcolor{black!75}{$N=62$}\\
\addlinespace[3pt]
 & PDQP.jl & 2923 & 15521 & 5890\\
 & PDHCG & 2152 & 5587 & 3212\\
 & HPR-QP-C & 2261 & 5330 & 3240\\
 & cuPDQPx & \textbf{1530} & \textbf{4215} & \textbf{2342}\\
\bottomrule
\end{tabular}
\par\smallskip
\begin{minipage}{\linewidth}
\footnotesize
$N$ is the number of common solved instances at the indicated tolerance.
Zero-iteration solves are included.
Bold indicates the lowest mean for each group and tolerance.
\end{minipage}
\end{table}

\end{document}